\documentclass{ws-idaqp}
\usepackage{graphicx}
\usepackage[super]{cite}
\renewcommand{\vec}[2]{{\rm vec}_{#1}(#2)}
\begin{document}

\markboth{Frederik vom Ende}{From the Choi Formalism in Infinite Dimensions to Unique Decompositions of Generators of Completely Positive Dynamical Semigroups}

\catchline{}{}{}{}{}

\title{FROM THE CHOI FORMALISM IN INFINITE DIMENSIONS TO UNIQUE DECOMPOSITIONS OF GENERATORS OF COMPLETELY POSITIVE DYNAMICAL SEMIGROUPS}

\author{FREDERIK VOM ENDE
}

\address{Dahlem Center for Complex Quantum
Systems, Freie Universität Berlin, Arnimallee 14\\
Berlin, 14195,
Germany\\
frederik.vom.ende@fu-berlin.de}

%

\maketitle

\begin{history}
\received{(Day Month Year)}
\revised{(Day Month Year)}
\published{(Day Month Year)}
\comby{(xxxxxxxxxx)}
\end{history}

\begin{abstract}
Given any separable complex Hilbert space, any trace-class operator $B$ which does not have purely imaginary trace, and any generator $L$ of a norm-continuous one-parameter semigroup of completely positive maps we prove that there exists a unique bounded operator $K$ and a unique completely positive map $\Phi$ such that (i) $L=K(\cdot)+(\cdot)K^*+\Phi$, (ii) the superoperator $\Phi(B^*(\cdot)B)$ is trace class and has vanishing trace, and (iii) ${\rm tr}(B^*K)$ is a real number. Central to our proof is a modified version of the Choi formalism which relates completely positive maps to positive semi-definite operators. We characterize when this correspondence is injective and surjective, respectively, which in turn explains why the proof idea of our main result cannot extend to non-separable Hilbert spaces. In particular, we find examples of positive semi-definite operators which have empty pre-image under the Choi formalism as soon as the underlying Hilbert space is infinite-dimensional.
\end{abstract}



\section{Introduction}
\label{sec_intro}

Completely positive maps sit at the heart of quantum information theory and irreversible quantum dynamics, the latter of which captures fundamental physical processes such as decoherence or measurements.
In particular, the evolution of many open quantum systems---that is, quantum systems which are not shielded from their environment---can be described by a norm-continuous one-parameter semigroup $(\Psi_t)_{t\geq0}$ of completely positive and trace-preserving
maps on a complex Hilbert space $\mathcal H$; this is also known as \textit{quantum-dynamical semigroup} \cite[Sec.~3.2]{BreuPetr02}.
The semigroup property together with continuity in the parameter $t$ guarantees that the whole evolution is fully captured by the generator $L$ of the semigroup, that is, the (unique) bounded operator $L$ such that $\Psi_t=e^{tL}$ for all $t\geq 0$, refer to \cite{EngelNagel00} for more detail.
What is more, there is even a standard form for generators of quantum-dynamical semigroups as first established in the seminal papers of Gorini et al.\cite{GKS76} as well as Lindblad \cite{Lindblad76}{}:
Every such $L$ can be written as\footnote{
Here, $[A,B]=AB-BA$ and $\{A,B\}=AB+BA$ are the usual commutator and anti-commutator, respectively.
Moreover, $\Phi^*$ is the dual of $\Phi$ which we will recap properly at the end of Section~\ref{sec_prelim}. All that is important for now is that the specific choice of operator in the anti-commutator guarantees that the generated semigroup is trace-preserving.
}
$-i[H,\cdot]+\Phi-\frac12\{\Phi^*({\bf1}),\cdot\}$ for some bounded, self-adjoint operator $H$ and some completely positive map $\Phi$; this is commonly known as \textit{GKSL-form}.
As trace-preservation 
amounts to the simple linear constraint ${\rm tr}(L(\rho))=0$ for all $\rho$,
as a slight generalization one finds that some $L$ generates a norm-continuous completely positive semigroup if and only if
$L=K(\cdot)+(\cdot)K^*+\Phi$ for some $K\in\mathcal B(\mathcal H)$ and some $\Phi$ completely positive \cite[Thm.~3.1]{CE79}{}.
From a physics perspective the term $-i[H,\cdot]$ in a GKSL-generator represents the intrinsic evolution of the system (according to the Liouville-von Neumann equation) whereas $\Phi-\frac12\{\Phi^*({\bf1}),\cdot\}$ models the interaction of the system with its surroundings.
Therefore, given some generator $L$, for applications and interpretation purposes it is desirable to know which part of the motion is due to the system itself and which part comes from the environment.
The more precise question here would be whether there exist some ``reasonable'' domain and co-domain\footnote{
Note that such domain considerations are inevitable as for general $K,\Phi$ one can always shift the Kraus operators (more on those in Sec.~\ref{sec_prelim_2}) of $\Phi$ by a multiple of the identity which leads to a modification of $H$ while leaving the overall generator $L$ invariant \cite[Eq.~1.4]{Davies80unique}.
}
such that the map $L\mapsto (H,\Phi)$ is well-defined.

At this point it should be mentioned that deriving quantum-dynamical semigroups physically, i.e.~obtaining the generator from microscopic considerations  \cite[Sec.~3.3]{BreuPetr02} usually comes with a ``natural'' GKSL-form already.
Hence the question of unique decompositions of GKSL-generators may---at least from a physics point-of-view---seem a bit artificial at first.
However, considering different decompositions of such generators is most useful, e.g., in equilibrium physics to obtain desirable properties such as quantum detailed balance \cite[Sec.~1.3.4]{AlickiLendi07} (we will make another comment about this connection later).
Moreover, different decompositions of $L$ also come up for dynamical semigroups with additional properties such as, e.g., covariance
\cite{Holevo93,SU23}{}.

Regardless of the physical motivation, the question of unique decompositions of $L$ is as old as the the GKSL-form itself:
In their original work Gorini et al.\cite{GKS76} have established that, for $\mathcal H$ finite-dimensional, such a unique decomposition is possible if both $H$ and $\Phi$ are traceless.
It has been shown that choosing $\Phi$ traceless corresponds to minimizing the dissipative term with respect to some ``average norm'' \cite{HS22}{}, which in turn has proven useful in studying open quantum systems from the perspective of quantum thermodynamics \cite{SS24}{}.
While this trace-zero condition has no meaningful counterpart beyond finite dimensions,
partial results in this direction have been achieved nonetheless:
Uniqueness in infinite dimensions has been established for special classes of generators, cf.~\cite{AF83} and for general $L$ under (rather restrictive) compactness conditions on the operator $K$, cf.~\cite{Davies80unique,FM92}{}.
Moreover,
Parthasarathy has derived a uniqueness result on the level of the Kraus operators $V_j$ (i.e.~the ``building blocks'') of $\Phi$ wherein he characterizes when different choices of $H$ and $V_j$ lead to the same generator $L$ \cite[Thm.~30.16]{Parthasarathy92}.
Note that Parthasarathy's result has been used to settle open questions regarding detailed balance of quantum Markov semigroups and their adjoint \cite{Fagnola07}{}.

Either way this paper's main result will improve upon---or at the very least simplify---the result of Parthasarathy, the central point being that we turn equivalence classes of $(H,\Phi)$ into a unique decomposition by (i) fixing the Hamiltonian $H$ via a trace condition 
and, more importantly, by (ii) eliminating the ambiguity of the Kraus operators by replacing Parthasarathy's trace condition (``${\rm tr}(B^*V_j)=0$'') with a trace condition on the level of the completely positive map $\Phi$.
Indeed, our main result reads as follows:
\medskip\smallskip

\noindent\textbf{Theorem~\ref{thm_main} (Informal).}
{\it
Let $L$ be the generator of a completely positive dynamical semigroup on a separable Hilbert space. Then for all trace-class operators $B$ with ${\rm Re}({\rm tr}(B))\neq 0$
there exists a unique bounded operator $K$ and a unique completely positive map $\Phi$ such that
\begin{itemize}
\item[(i)] the map $\Phi(B^*(\cdot)B)$ has trace zero (equivalently: for every set of Kraus operators $\{V_j\}_j$ of $\Phi$, it holds that ${\rm tr}(B^*V_j)=0$ for all $j$),
\item[(ii)] ${\rm Im}({\rm tr}(B^*K))=0$, and
\item[(iii)] $L=K(\cdot)+(\cdot)K^*+\Phi$.
\end{itemize}
}\smallskip

\noindent For generators of GKSL-form and self-adjoint $B$ condition (ii) reduces to the vanishing expectation value ${\rm tr}(BH)=0$.
%
Either way 
this extends a finite-dimensional result of ours
\cite{vE24_decomp_findim}{},
a key tool of which was the Choi matrix, that is, the matrix $\mathsf C(\Phi):=({\rm id}\otimes\Phi)(|\Gamma\rangle\langle\Gamma|)$ where $\Phi$ is an arbitrary linear map,
$\Gamma:=\sum_{j=1}^n |j\rangle\otimes|j\rangle\in\mathbb C^n\otimes\mathbb C^n$ is the unnormalized entangled state, and $|j\rangle$ is the $j$-th standard basis vector.
This formalism is used to relate completely positive maps to positive semi-definite operators on an enlarged Hilbert space (cf.~Section~\ref{prelim_choi} for more details).
Thus, trying to prove the above main result inevitably raises a number of questions and issues due to the infinite-dimensional setting:
1.~%
It is not immediately clear when (resp.~under which conditions) the trace of $\Phi(B^*(\cdot)B)$ even exists.
2.~What approaches to the Choi matrix in infinite dimensions exist, and which of them (if any) are suited to help prove our main result?
%
We will address these problems in this work
and in the process we will obtain new results on the infinite-dimensional Choi formalism as well as on when linear maps between Schatten-class operators are themselves part of a Schatten class.

This work is structured as follows.
In Section~\ref{sec_prelim} we recall some basic facts on Schatten classes as well as operators thereon.
In Section~\ref{prelim_choi} we review different approaches to the Choi formalism for general Hilbert spaces, and we prove a new characterization for when 
the most common formalism is injective and surjective (Proposition~\ref{prop_choi_infdim}); the upshot there is that surjectivity never holds in infinite dimensions 
which is proven by considering certain reset maps $X\mapsto{\rm tr}(AX)B$. 
Then, in Section~\ref{sec_prelim_2} we focus on Schatten-class operators which themselves act on Schatten classes. Important new results shown therein are conditions on $\Phi$ and $B$ such that the map $X\mapsto\Phi(B^*XB)$ is trace class (Lemma~\ref{lemma_weighted_Phi_welldef}), and that for $\Phi$ completely positive the trace of $X\mapsto\Phi(B^*XB)$ can be computed explicitly via the Kraus operators of $\Phi$ (Lemma~\ref{lemma_v_ell2}).
Finally, Section~\ref{sec_uniquedecomp} is dedicated to 
our main result (Theorem~\ref{thm_main}).
As a special case we obtain (a family of) unique decompositions for generators of quantum-dynamical semigroups (Corollary~\ref{coro_unique_decomp_qds}).

\section{Preliminaries: Schatten Classes and Tensor Products}\label{sec_prelim}

We begin by quickly recapping some (notation from) operator theory in general and Schatten classes in particular.
Unless specified otherwise, $\mathcal H,\mathcal Z$ will---here and henceforth---denote arbitrary complex Hilbert spaces. 
Be aware that as our Hilbert spaces may be non-separable
we need to invoke nets as well as the concept of unordered summation in order to 
properly address questions of convergence\footnote{
Following \cite{Ringrose71}{}---or, alternatively, \cite[Ch.~12]{MeiseVogt97en}---recall that a
mapping $f:J\to \mathcal X$ from a non-empty set $J$ into a real or complex normed 
space $\mathcal X$ is called \textit{summable} (to $x\in\mathcal X$) if
the net $\{\sum_{j\in F}f(j)\}_{F\subseteq J\text{ finite}}$ norm-converges (to $x$).
}.
Moreover, as is standard in 
quantum information theory
$|x\rangle\langle y|$
will be short for the linear operator $z\mapsto \langle y,z\rangle x$.
This is also why, sometimes, we will write $|x\rangle$ instead of $x$.

With this, our notation for common operator spaces reads as follows:
$\mathcal L(\mathcal H,\mathcal Z)$ is the vector space of all linear maps $:\mathcal H\to\mathcal Z$,
while
$\mathcal B(\mathcal H,\mathcal Z)$ is the Banach space of all bounded linear maps (where $\|\cdot\|_\infty:=\sup_{x\in\mathcal H,\|x\|=1}\|(\cdot)x\|$ denotes the usual operator norm),
and $\mathcal K(\mathcal H,\mathcal Z)$ is the subspace of all compact maps (i.e.~linear maps such that the closure of the image of the
closed unit ball is compact);
these notions obviously generalize from Hilbert to Banach spaces.
The final step in this chain is to go to
Schatten classes:
Following
\cite{MeiseVogt97en}{}, \cite{Ringrose71}{}, or \cite{Dunford63}
every $X\in\mathcal K(\mathcal H,\mathcal Z)$ can be written as $X=\sum_{j\in N}s_j(X)|f_j\rangle\langle g_j|$ for some $N\subseteq\mathbb N$, some orthonormal systems $\{f_j\}_{j\in N},\{g_j\}_{j\in N}$ of $\mathcal Z$, $\mathcal H$, respectively, and a (unique) decreasing null sequence $\{s_j(X)\}_{j\in N}$.
This is known as \textit{Schmidt decomposition} of $X$ and 
the $s_j(X)$ are sometimes called \textit{singular values}.
Then, given any $p>0$ one defines the Schatten-$p$ class $\mathcal B^p(\mathcal H,\mathcal Z):=\{X\in\mathcal K(\mathcal H,\mathcal Z):\sum_js_j(X)^p<\infty\}$
with
corresponding Schatten ``norm''\,\footnote{
Note that---like for $\ell^p$-spaces---$\|\cdot\|_p$ is a norm if and only if $p\in[1,\infty]$ in which case $\mathcal B^p(\mathcal H,\mathcal Z)$ is even a Banach space \cite[Coro.~16.34]{MeiseVogt97en}.
}
$\|X\|_p:=(\sum_js_j(X)^p)^{1/p}
$.
Here one usually identifies $\mathcal B^\infty(\mathcal H,\mathcal Z):=\mathcal K(\mathcal H,\mathcal Z)$.
In particular, $|x\rangle\langle y|$ for all $x,y\in\mathcal H$ is in $\mathcal B^p(\mathcal H)$ for all $p>0$ because $\|\,|x\rangle\langle y|\,\|_p=\|x\|\|y\|$.
Recall the following basic (composition) rules of Schatten classes:
\begin{lemma}\label{lemma_schatten_comp}
Given complex Hilbert spaces $\mathcal H_1,\mathcal H_2,\mathcal H_3,\mathcal H_4$ and any $p,q,r>0$,
the following statements hold.
\begin{itemize}
\item[(i)] If $\frac1r=\frac1p+\frac1q$, then for all $X\in\mathcal B^p(\mathcal H_2,\mathcal H_3)$, $Y\in\mathcal B^q(\mathcal H_1,\mathcal H_2)$ one has $XY\in\mathcal B^r(\mathcal H_1,\mathcal H_3)$.
Moreover, if $p,q,r\geq 1$, then $\|XY\|_r\leq\|X\|_p\|Y\|_q$.
\item[(ii)] If $\frac1r=\frac1p+\frac1q$, then for all $X\in\mathcal B^r(\mathcal H_1,\mathcal H_2)$ there exist 
$Y\in\mathcal B^p(\mathcal H_2)$ and $Z\in \mathcal B^q(\mathcal H_1,\mathcal H_2)$ such that $X=YZ$.
\item[(iii)] If $q\geq p$, then for all $X\in\mathcal B^p(\mathcal H_1,\mathcal H_2)$, it holds that $\|X\|_q\leq\|X\|_p$.
In particular, $\mathcal B^p(\mathcal H_1,\mathcal H_2)\subseteq\mathcal B^q(\mathcal H_1,\mathcal H_2)$ whenever $q\geq p$.
\item[(iv)] For all 
$X\in\mathcal B(\mathcal H_3,\mathcal H_4)$, $Z\in\mathcal B(\mathcal H_1,\mathcal H_2)$, $Y\in\mathcal B^p(\mathcal H_2,\mathcal H_3)$
one has $XYZ\in\mathcal B^p(\mathcal H_1,\mathcal H_4)$ with
$\|XYZ\|_p\leq\|X\|\|Y\|_p\|Z\|$.
\end{itemize}
\end{lemma}
\begin{proof}
We will only prove those statements which can not be found in 
the literature referred to above.
(ii): By assumption 
$X$
admits a Schmidt decomposition $\sum_{j\in N}s_j(X)|f_j\rangle\langle g_j|$.
Then $Y:=\sum_{j\in N}(s_j(X))^{r/p}|f_j\rangle\langle f_j|$ and $Z:=\sum_{j\in N}(s_j(X))^{r/q}|f_j\rangle\langle g_j|$ do the job.
(iii): This result is stated in \cite[Ch.~XI.9, Lemma 9]{Dunford63} but without a proof, so we will fill this gap for the reader's convenience.
Let $X\in \mathcal B^p(\mathcal H_1,\mathcal H_2)$ and w.l.o.g.~$X\neq 0$.
We have to show $\| \,X\|X\|_p^{-1} \,\|_q\leq 1$.
Defining $X':=X\|X\|_p^{-1}$ (i.e.~$\|X'\|_p=1$) one finds $s_j(X')\leq 1$ for all $j\in N$; in particular this shows $(s_j(X'))^{q/p}\leq s_j(X')$ (as $\frac{q}{p}\geq 1$).
This, as desired, implies $\|X'\|_q\leq 1$ due to $\|X'\|_q^q=\sum_{j\in N}(s_j(X'))^q\leq\sum_{j\in N}(s_j(X'))^p=1$.
\end{proof}

Important Schatten classes are the trace class $\mathcal B^1(\mathcal H,\mathcal Z)$---where, assuming $\mathcal H=\mathcal Z$, the trace
${\rm tr}(A):=\sum_{j\in J}\langle g_j,Ag_j\rangle$ is well defined and independent of the chosen orthonormal basis $\{g_j\}_{j\in J}$ of $\mathcal H$---as well as
the Hilbert-Schmidt class $\mathcal B^2(\mathcal H,\mathcal Z)$ which is itself a
Hilbert space with respect to the inner product $\langle X,Y\rangle_{\sf HS}:={\rm tr}(X^*Y)$.
A well-known, yet important fact is that an orthonormal basis of $\mathcal B^2(\mathcal H,\mathcal Z)$ is given by $\{|f_k\rangle\langle g_j|\}_{j\in J,k\in K}$
where $\{f_k\}_{k\in K}$, $\{g_j\}_{j\in J}$ are arbitrary orthonormal bases
of $\mathcal Z$, $\mathcal H$, respectively\footnote{
The idea---which I include because I could not find a reference which covers this result for non-separable Hilbert spaces---is that the span of the (obviously orthogonal) set $\{|f_k\rangle\langle g_j|\}_{j\in J,k\in K}$ is
dense in $\mathcal B^2(\mathcal H,\mathcal Z)$ (actually: dense in $\mathcal B^p$ for all $p\in[1,\infty]$, cf.~Coro.~\ref{coro_block_approx_nonsep} in 
the appendix).
Hence it is an orthonormal basis \cite[Thm.~1.6.3]{Ringrose71}.
}.

Another important concept we need in this regard is the tensor product of Hilbert spaces and of operators thereon.
We will write $\mathcal H\otimes\mathcal Z$ for the (Hilbert space) tensor product of $\mathcal H,\mathcal Z$;
then, given any orthonormal bases $\{g_j\}_{j\in J}$, $\{f_k\}_{k\in K}$ of $\mathcal H$, $\mathcal Z$, respectively, $\{g_j\otimes f_k\}_{j\in J,k\in K}$ is an orthonormal basis of $\mathcal H\otimes\mathcal Z$.
On the level of operators, $\mathcal B^2(\mathcal H\otimes\mathcal Z)$
is isometrically isomorphic (in the sense of Hilbert spaces) to 
$\mathcal B^2(\mathcal H)\otimes\mathcal B^2(\mathcal Z)$ for all complex 
Hilbert spaces $\mathcal H,\mathcal Z$.
This is due to the fact that 
$\{|f_j\rangle\langle f_k|\otimes|g_a\rangle\langle g_b|\}_{j,k\in J,a,b\in A}=\{|f_j\otimes g_a\rangle\langle f_k\otimes g_b|\}_{j,k\in J,a,b\in A}$ 
is an orthonormal basis for both of these spaces, where $\{f_j\}_{j\in J}$, $\{g_a\}_{a\in A}$ is any orthonormal basis of $\mathcal H$, $\mathcal Z$, respectively.
For more details we refer to \cite[Ch.~2.6]{Kadison83} or \cite[Appendix~A.3]{vE_PhD_2020}.

\paragraph*{Operators on Schatten Classes}
The objects central to our main results are linear maps between operator spaces (sometimes called \textit{superoperators}).
First, given any $\Phi\in\mathcal B(\mathcal B^p(\mathcal H),\mathcal B^q(\mathcal Z))$, $p,q\geq 1$ we denote the operator norm $\sup_{A\in\mathcal B^p(\mathcal H),\|A\|_p=1}\|\Phi(A)\|_q$ by $\|\Phi\|_{p\to q}$\,; 
similarly we write $\|\Phi\|_{\infty\to\infty}$ for the norm of any $\Phi\in\mathcal B(\mathcal B(\mathcal H),\mathcal B(\mathcal Z))$.
Next, some $\Phi\in\mathcal L(\mathcal B^1(\mathcal H),\mathcal B^1(\mathcal Z))$ is called
\begin{itemize}
\item \textit{positive} if for all $A\in\mathcal B^1(\mathcal H)$ positive semi-definite (i.e.~$A=A^*$
and
$\langle x,Ax\rangle\geq 0$ for all $x\in\mathcal H$; usually denoted by
$A\geq 0$)
one has $\Phi(A)\geq 0$.
\item \textit{$n$-positive} for some $n\in\mathbb N$, if ${\rm id}_n\otimes\Phi:\mathcal B^1(\mathbb C^n\otimes\mathcal H)\to\mathcal B^1(\mathbb C^n\otimes\mathcal Z)$ defined via\footnote{
Here one uses implicitly that $\mathbb C^n\otimes\mathcal H\simeq\mathcal H\times\ldots\times\mathcal H$ \cite[Rem.~2.6.8]{Kadison83}
so $\mathcal B(\mathbb C^n\otimes\mathcal H)$ can be identified with $\mathbb C^{n\times n}\otimes\mathcal B(\mathcal H)$ \cite[p.~147 ff.]{Kadison83}, and similarly for the trace class.
}
\begin{align*}
({\rm id}_n\otimes\Phi)\begin{pmatrix}
A_{11}&\cdots&A_{1n}\\
\vdots&\ddots&\vdots \\
A_{n1}&\cdots&A_{nn}
\end{pmatrix}:=\begin{pmatrix}
\Phi(A_{11})&\cdots&\Phi(A_{1n})\\
\vdots&\ddots&\vdots \\
\Phi(A_{n1})&\cdots&\Phi(A_{nn})
\end{pmatrix}
\end{align*}
for all $\{A_{jk}\}_{j,k=1}^n\subset\mathcal B^1(\mathcal H)$ is positive for all $n\in\mathbb N$.
\item \textit{completely positive} if $\Phi$ is $n$-positive for all $n\in\mathbb N$.
We denote the set of all completely positive maps $\Phi:\mathcal B^1(\mathcal H)\to\mathcal B^1(\mathcal Z)$ by $\mathsf{CP}(\mathcal H,\mathcal Z)$.
\end{itemize}
\noindent The notions for maps $\Phi\in\mathcal L(\mathcal B(\mathcal H),\mathcal B(\mathcal Z))$ are analogous by means of the isometric isomorphism
$\mathcal B(\mathcal Z,\mathcal H)\simeq (\mathcal B^1(\mathcal H,\mathcal Z))'$, $B\mapsto {\rm tr}(B(\cdot))$.
Indeed, this not only translates 
the weak${}^*$-topology on $(\mathcal B^1(\mathcal H,\mathcal Z))'$ into a topology on $\mathcal B(\mathcal Z,\mathcal H)$---called the \textit{ultraweak topology}\,\footnote{
i.e.~a net $\{B_j\}_{j\in J}\subseteq\mathcal B(\mathcal Z,\mathcal H)$ converges to $B\in\mathcal B(\mathcal Z,\mathcal H)$ in the ultraweak topology (``ultraweakly'') if and only if $\{{\rm tr}(B_jA)\}_{j\in J}$ converges to ${\rm tr}(BA)$ for all $A\in \mathcal B^1(\mathcal H,\mathcal Z)$.
}---but
every $\Phi\in\mathcal B(\mathcal B^1(\mathcal H),\mathcal B^1(\mathcal Z))$ induces a unique dual map $\Phi^*\in\mathcal B(\mathcal B(\mathcal Z),\mathcal B(\mathcal H))$ via ${\rm tr}(\Phi(A)B)={\rm tr}(A\Phi^*(B))$ for all $A\in\mathcal B^1(\mathcal H)$, $B\in\mathcal B(\mathcal Z)$.
Then $\|\Phi\|_{1\to 1}=\|\Phi^*\|_{\infty\to\infty}$ and
(complete) positivity of $\Phi$ is 
well known to be equivalent to (complete) positivity of $\Phi^*$.

\section{Recap: The Choi Matrix in Infinite Dimensions}\label{prelim_choi}

As discussed before we need to 
make sense of the Choi matrix $({\rm id}\otimes\Phi)(|\Gamma\rangle\langle\Gamma|)$ in infinite dimensions, and there are different ways to go about this.
Two approaches where 
none of the involved objects get modified are as follows:
One can either consider the quadratic form induced by the Choi matrix---called
Choi-Jamio\l{}kowski form---cf.~\cite{Holevo11b,Holevo11,Haapasalo21}{},
or one can take appropriate inductive limits 
by considering finite truncations of $|\Gamma\rangle\langle\Gamma|$ as well as of the output, 
cf.~\cite{Friedland19}{}.
For our purposes, however, we can adopt a more naive method where one weights the input of the ``usual'' Choi matrix, i.e.~one replaces $|\Gamma\rangle$ by $\Gamma_{\lambda,G}:=\sum_{j\in J}\lambda_j^* g_j\otimes g_j\in\mathcal H\otimes\mathcal H$ with
$\lambda\in\ell^2(J,\mathbb C)$ where $G:=\{g_j\}_{j\in J}$ is any orthonormal basis of $\mathcal H$, cf.~\cite{LiDu15}{}.
The state $\Gamma_{\lambda,G}$ is also known
as two-mode squeezed vacuum state, cf.~\cite{SC85}{}
which is why 
this approach
being widely adopted in, e.g., quantum optics, cf.~\cite{PLOB17,KBUP17} as well as \cite[Ch.~5.2]{Serafini17}.
One drawback of this weighting approach is that for general $\Phi\in\mathcal B(\mathcal B^1(\mathcal H))$ the object $({\rm id}\otimes\Phi)(|\Gamma_{\lambda,G}\rangle\langle \Gamma_{\lambda,G}|)$ may be problematic:
Indeed, if $\Phi$ is the transposition map (w.r.t.~an arbitrary but fixed orthonormal basis), then $\Phi$ is positive and trace-preserving but the corresponding operator $({\rm id}\otimes\Phi)(|\Gamma_{\lambda,G}\rangle\langle \Gamma_{\lambda,G}|)$ is only densely defined, cf.~\cite{Tomiyama83,Paulsen03}{}.

The first option for guaranteeing existence of $({\rm id}\otimes\Phi)(|\Gamma_{\lambda,G}\rangle\langle \Gamma_{\lambda,G}|)$ is to restrict
$\Phi$ to completely bounded maps\footnote{A map $\Phi\in\mathcal B(\mathcal B^1(\mathcal H))$ is called completely bounded if $\sup_{n\in\mathbb N}\|{\rm id}_n\otimes\Phi\|_{1\to 1}<\infty$.}, cf.~\cite{Stormer15,Magajna21,HKS24}{}.
This is usually done in the framework of von Neumann algebras and factors where the explicit vector $\Gamma_{\lambda,G}$ from above is replaced by an abstract separating and cyclic vector of the factor at hand (in our case: ${\bf1}\otimes\mathcal B(\mathcal H)$).
However, while every completely positive map is completely bounded 
\cite[Prop.~3.6]{Paulsen03},
the completely bounded maps are nowhere dense in $\mathcal B(\mathcal B^1(\mathcal H))$ 
\cite[Thm.~2.4 \& 2.5]{Smith83}.
In contrast, the second way to make sense of the Choi formalism is the one of \cite{LiDu15}
where one defines the Choi operator via an appropriate infinite sum.
This is also the route we will take as it is precisely what
we will need when proving our main result in Section~\ref{sec_uniquedecomp}.
Because we are allowing for non-separable Hilbert spaces we also provide a sketch of the proof for the following lemma.
\begin{lemma}\label{lemma_choi}
Given
any orthonormal basis $G:=\{g_j\}_{j\in J}$ of $\mathcal H$
and any $\lambda\in\ell^2(J,\mathbb C)$
\begin{align*}
\mathsf C_{\lambda,G}:\mathcal B(\mathcal B^1(\mathcal H),\mathcal B^1(\mathcal Z))&\to \mathcal B^2(\mathcal H\otimes\mathcal Z)\\
\Phi&\mapsto \sum_{j,k\in J}\lambda_j^*\lambda_k|g_j\rangle\langle g_k|\otimes\Phi(|g_j\rangle\langle g_k|)
\end{align*}
is well-defined, linear, and bounded with $\|\mathsf C_{\lambda,G}\|\leq\|\lambda\|_{2}^2$.
Moreover,
$\mathsf C_{\lambda,G}(\Phi)$ is positive semi-definite
whenever $\Phi\in\mathsf{CP}(\mathcal H,\mathcal Z)$.
\end{lemma}
\begin{proof}
The key observation is that $\{\lambda_j^*\lambda_k|g_j\rangle\langle g_k|\otimes\Phi(|g_j\rangle\langle g_k|):(j,k)\in J\times J\}$ is an orthogonal subset of $\mathcal B^2(\mathcal H\otimes\mathcal Z)$
meaning it is summable
if and only if $\sum_{j,k\in J}\|\lambda_j^*\lambda_k|g_j\rangle\langle g_k|\otimes\Phi(|g_j\rangle\langle g_k|)\|_2^2<\infty$ \cite[Lemma~1.6.1]{Ringrose71}.
Using Lemma~\ref{lemma_schatten_comp}~(iii) the latter sum evaluates to
\begin{align*}
\sum_{j,k\in J}\|\lambda_j^*\lambda_k|g_j\rangle\langle g_k|\otimes\Phi(|g_j\rangle\langle g_k|)\|_2^2&=\sum_{j,k\in J}|\lambda_j|^2|\lambda_k|^2\,\big\||g_j\rangle\langle g_k|\big\|_2^2\big\|\Phi(|g_j\rangle\langle g_k|)\big\|_2^2\\
&\leq \sum_{j,k\in J}|\lambda_j|^2|\lambda_k|^2\big\|\Phi(|g_j\rangle\langle g_k|)\big\|_1^2\\
&\leq\|\Phi\|_{1\to 1}^2\|\lambda\|_2^4<\infty\,.
\end{align*}
Hence $\mathsf C_{\lambda,G}$ is well defined, linear, and (again by \cite[Lemma~1.6.1]{Ringrose71}) bounded with $\|\mathsf C_{\lambda,G}(\Phi)\|_{2}^2\leq\|\Phi\|_{1\to 1}^2\|\lambda\|_{2}^4$, i.e.~$\|\mathsf C_{\lambda,G}\|\leq\|\lambda\|_{2}^2$.
For the final statement, given $F\subseteq J$ non-empty and finite define $\psi_F:=\sum_{j\in F}\lambda_j^* |j\rangle\otimes g_j\in\mathbb C^{|F|}\otimes\mathcal H$ and \mbox{$U_F:{\rm span}\{g_j:j\in F\}\to\mathbb C^{|F|}$} as the unique linear map such that $Ug_j=|j\rangle$ for all $j\in F$.
In particular,
$U_F$ is unitary.
With this, one readily verifies 
$\sum_{j,k\in F}\lambda_j^*\lambda_k|g_j\rangle\langle g_k|\otimes\Phi(|g_j\rangle\langle g_k|)=(U_F\otimes{\bf1}_{\mathcal H})^*( ({\rm id}_{|F|}\otimes\Phi)(|\psi_F\rangle\langle\psi_F|)) (U_F\otimes{\bf1}_{\mathcal H})$.
But the latter is $\geq 0$
because $\Phi$ is completely positive; thus the same holds for
the
weak
limit $\mathsf C_{\lambda,G}(\Phi)$ of $\{\sum_{j,k\in F}\lambda_j^*\lambda_k|g_j\rangle\langle g_k|\otimes\Phi(|g_j\rangle\langle g_k|)\}_{F\subseteq J\text{ finite}}$,
because
for any $X\in\mathcal B^2(\mathcal H)$, $X\geq 0$ 
if and only if ${\rm tr}(BX)\geq 0$ for all $B\in\mathcal B^2(\mathcal H)$, $B\geq 0$.
\end{proof}
We waived the converse of this lemma (i.e.~Choi operator being positive implies complete positivity, provided $\lambda_j\neq0$ for all $j\in J$) because we will not need it in this work, and because $\lambda\in\ell^2(J,\mathbb C\setminus\{0\})$ only makes sense in the separable case anyway which is already covered, e.g., by \cite[Thm.~1.4]{LiDu15}.

The feature of the Choi formalism
most
important 
to quantum information theory
is that it establishes a one-to-one relation between completely positive trace-preserving linear maps---also known as ``CPTP maps'' or ``quantum channels''---and quantum states (positive semi-definite trace-class operators of unit trace) on the larger space $\mathcal H\otimes\mathcal H$ which satisfy\footnote{
In what follows, given any $A\in\mathcal B^1(\mathcal H\otimes\mathcal Z)$ we write ${\rm tr}_{\mathcal H}(A)$
for the unique operator in $\mathcal B^1(\mathcal Z)$ which satisfies ${\rm tr}({\rm tr}_{\mathcal H}(A)B)={\rm tr}(A({\bf1}_{\mathcal H}\otimes B))$ for all $B\in\mathcal B(\mathcal Z)$.
\label{footnote_part_trace}
}
${\rm tr}_{\mathcal H}(\rho)=\frac{{\bf1}_{\mathcal Z}}{\dim(\mathcal Z)}$ \cite[Thm.~4.48]{Heinosaari12}.
More generally, the one-to-one correspondence in finite dimensions is between completely positive maps and positive semi-definite matrices, assuming $\lambda_j\neq 0$ for all $j$.
There are two possible approaches of extending this to infinite dimensions:
either one restricts the domain of $\mathsf C_{\lambda,G}$ to the completely bounded maps, or one restricts the sequence $\lambda$ to
something absolutely summable.
Indeed, the previously discussed transposition map example shows that $\lambda\in\ell^1(J,\mathbb C)$ is the ``best possible choice'': for no $\lambda\in\ell^p(J,\mathbb C)\setminus\ell^1(J,\mathbb C)$ and no $p>1$ 
would $\mathsf C_{\lambda,G}$ (with co-domain $\mathcal B^1$) be well defined.
As the path via completely bounded maps has been sufficiently explored already (more on this in a bit) we will pursue the $\ell^1$-approach.
Doing so---like in Lemma~\ref{lemma_choi}---yields a well-defined map
\begin{equation}\label{eq:choi_B1}
\begin{split}
\mathsf C_{\lambda,G}:\mathcal B(\mathcal B^1(\mathcal H),\mathcal B^1(\mathcal Z))&\to \mathcal B^1(\mathcal H\otimes\mathcal Z)\\
\Phi&\mapsto \sum_{j,k\in J}\lambda_j^*\lambda_k|g_j\rangle\langle g_k|\otimes\Phi(|g_j\rangle\langle g_k|)\,.
\end{split}
\end{equation}
However, even with this modification in place it turns out that the channel-state (and even the CP-PSD) duality of the Choi formalism is a purely finite-dimensional effect.
More precisely, injectivity of $\mathsf C_{\lambda,G}$ needs separability of $\mathcal H$, and surjectivity never holds as soon as $\mathcal H$ is infinite dimensional.
This complements recent, similar results for the completely bounded approach \cite[Thms.~2.2 \& 3.3]{HKS24},
and this will be an important insight when discussing possible generalizations of our main result later on (Sec.~\ref{sec_open_q}).

\begin{proposition}\label{prop_choi_infdim}
Given any complex Hilbert spaces $\mathcal H,\mathcal Z$, any orthonormal basis $G:=\{g_j\}_{j\in J}$ of $\mathcal H$,
and any $\lambda\in\ell^1(J,\mathbb C)$
the map $\mathsf C_{\lambda,G}$ from Eq.~\eqref{eq:choi_B1} is
\begin{itemize}
\item[(i)]
injective if and only if $\mathcal H$ is separable and $\lambda_j\neq 0$ for all $j\in J$.
\item[(ii)]
surjective if and only if $\mathcal H$ is finite-dimensional and $\lambda_j\neq 0$ for all $j\in J$.
\end{itemize}
In particular, if $\dim\mathcal H=\infty$, then there exist positive semi-definite trace-class operators on $\mathcal H\otimes\mathcal Z$ which have empty pre-image under $\mathsf C_{\lambda,G}$, regardless of the chosen $\lambda\in\ell^1(J,\mathbb C)$ and the chosen orthonormal basis $\{g_j\}_{j\in J}$ of $\mathcal H$.
\end{proposition}
\begin{proof}
(i): ``$\Rightarrow$'': If $\mathcal H$ is not separable, then $J$ is uncountable meaning
$\lambda_j=0$ for some $j\in J$ \cite[Lemma~1.2.7]{Ringrose71}.
Thus it suffices to show that $\mathsf C_{\lambda,G}$ has non-trivial kernel 
as soon as $\lambda_j=0$ for some $j\in J$.
For this note that $\Phi_j\in\mathcal B(\mathcal B^1(\mathcal H),\mathcal B^1(\mathcal Z))$ defined via $\Phi_j(X):=\langle g_j,Xg_j\rangle Z$ (where $Z\in\mathcal B^1(\mathcal Z)\setminus\{0\}$ is arbitrary but fixed, so $\Phi_j\neq 0$)
is in the kernel of $\mathsf C_{\lambda,G}$
due to $\mathsf C_{\lambda,G}(\Phi_j)= |\lambda_j|^2|g_j\rangle\langle g_j|\otimes Z=0 $.

``$\Leftarrow$'': Given two elements $\Phi_1\neq \Phi_2$ from $\mathcal B(\mathcal B^1(\mathcal H),\mathcal B^1(\mathcal Z))$
there have to exist $j,k\in J$ such
that 
$\Phi_1(|g_j\rangle\langle g_k|)\neq \Phi_2(|g_j\rangle\langle g_k|)$ for some $j,k\in J$ (because $\operatorname{span}\{|g_j\rangle\langle g_k|:j,k\in J\}$
is dense in $(\mathcal B^1(\mathcal H),\|\cdot\|_1)$, cf.~Coro.~\ref{coro_block_approx_nonsep} in the appendix). Hence for some $x,y\in\mathcal Z$
\begin{align*}
(\lambda_j^*\lambda_k)^{-1}\langle g_j\otimes x, \mathrm{C}_\lambda(\Phi_1) (g_k\otimes y)\rangle&=
\langle x,\Phi_1(|g_j\rangle\langle g_k|)y\rangle\neq\langle x,\Phi_2(|g_j\rangle\langle g_k|)y\rangle\\
&=(\lambda_j^*\lambda_k)^{-1}\langle g_j\otimes x,\mathrm{C}_\lambda(\Phi_2)  (g_k\otimes y)\rangle\,.
\end{align*}
Thus the assumption $\lambda_j,\lambda_k\neq 0$ implies $\mathrm{C}_\lambda(\Phi_1)\neq\mathrm{C}_\lambda(\Phi_2)$ as desired.

(ii): ``$\Leftarrow$'': Let $X\in\mathcal B^1(\mathcal H\otimes\mathcal Z)$. Because $\lambda_j\neq 0$ for all $j\in J$ and because $\dim\mathcal H<\infty$ 
one may define
a unique (bounded, because finite-dimensional domain) linear map $\Phi_X:\mathcal B^1(\mathcal H)=\mathbb C^{|J|\times|J|}\to\mathcal B^1(\mathcal Z)$ via the relation
$
\Phi_X(|g_j\rangle\langle g_k|)=(\lambda_j^*\lambda_k)^{-1}{\rm tr}_{\mathcal H}((|g_k\rangle\langle g_j|\otimes{\bf1}_{\mathcal Z})X)
$
for all $j,k\in J$ (also recall footnote~\ref{footnote_part_trace}).
A straightforward computation then shows
$\mathsf C_{\lambda,G}(\Phi_X)=X$.

``$\Rightarrow$'':
We argue by contraposition, so we have to take care of two cases:
\begin{itemize}
\item[1.] If there exists $j\in J$ such that $\lambda_j=0$, then $\mathsf C_{\lambda,G}^{-1}(|g_j\rangle\langle g_j|\otimes Z)=\emptyset$ for arbitrary but fixed $Z\in\mathcal B^1(\mathcal H)\setminus\{0\}$ (which shows that $\mathsf C_{\lambda,G}^{-1}$ is not surjective):
Assume to the contrary that there exists $\Phi\in\mathcal B(\mathcal B^1(\mathcal H),\mathcal B^1(\mathcal Z))$ such that $\mathsf C_{\lambda,G}(\Phi)=|g_j\rangle\langle g_j|\otimes Z$.
Given any $x,y\in\mathcal Z$ such that $\langle x,Z y\rangle\neq 0$---which exist because $Z\neq 0$---we compute
\begin{align*}
0\neq\langle x,Z y\rangle&=\big\langle  g_j\otimes x, (|g_j\rangle\langle g_j|\otimes Z)( g_j\otimes y )\big\rangle\\
&=\big\langle  g_j\otimes x, \mathsf C_{\lambda,G}(\Phi)( g_j\otimes y )\big\rangle=|\lambda_j|^2\langle x, \Phi(  |g_j\rangle\langle g_j|  ) y\rangle\,,
\end{align*}
contradicting $\lambda_j= 0$.
\item[2.] 
Assume $\dim{\mathcal H}=\infty$ and
w.l.o.g.~$\mathbb N\subseteq J$.
Given any
$z\in\mathcal Z\setminus\{0\}$
we claim that $Y:=(\sum_{j\in\mathbb N}\frac1{j^2}|g_j\rangle\langle g_j|)\otimes|z\rangle\langle z|\in\mathcal B^1(\mathcal H\otimes\mathcal Z)$ has empty pre-image under $\mathsf C_{\lambda,G}$,
regardless of the chosen $G,\lambda$.
Assume to the contrary that there exists $\Phi\in\mathcal B(\mathcal B^1(\mathcal H),\mathcal B^1(\mathcal Z))$ such that $\mathsf C_{\lambda,G}(\Phi)=Y$.
This for all $j\in\mathbb N$ implies
\begin{align*}
\|z\|^2j^{-2}&=\langle g_j\otimes z,Y(g_j\otimes z)\rangle\\
&=\langle g_j\otimes z, \mathsf C_{\lambda,G}(\Phi) (g_j\otimes z)\rangle=|\lambda_j|^2 \langle z,  \Phi(|g_j\rangle\langle g_j|)  z\rangle 
\end{align*}
so in particular $\lambda_j\neq 0$ for all $j\in\mathbb N$.
Together with boundedness of $\Phi$ this yields
$
\infty>\|\Phi\|_{1\to 1}\|z\|^2\geq\sup_{j\in\mathbb N}| \langle z,  \Phi(|g_j\rangle\langle g_j|) z\rangle|=\sup_{j\in\mathbb N}\tfrac{\|z\|^2}{j^2|\lambda_j|^2}
$,
i.e.~there exists $C>0$ such that $\frac{1}{j^2|\lambda_j|^2}\leq C$ for all $j\in\mathbb N$.
Therefore
$
\infty=C^{-1/2}\sum_{j\in\mathbb N}\frac1j\leq \sum_{j\in\mathbb N}|\lambda_j|\leq \|\lambda\|_1
$,
a contradiction.
\end{itemize}
The second case in the proof of (ii) also proves the additional statement: $Y$ is positive semi-definite, but its pre-image under $\mathsf C_{\lambda,G}$ would \textit{formally} read $X\mapsto{\rm tr}(\Lambda X)|z\rangle\langle z|$
where $\Lambda:=\sum_{j\in\mathbb N}(j|\lambda_j|)^{-2}|g_j\rangle\langle g_j|$.
However, $\lambda\in\ell^1(J,\mathbb C)$ forces $\Lambda$ to be unbounded meaning the formal map $X\mapsto{\rm tr}(\Lambda X)|z\rangle\langle z|$ is not well defined if and only if ${\rm dim}(\mathcal H)=\infty$.
\end{proof}
\noindent 
Note that lack of surjectivity does \textit{not} come from our choice to go from $\lambda\in\ell^2(J,\mathbb C)$ to $\ell^1(J,\mathbb C)$.
This can be seen by adjusting the argument from the proof of Proposition~\ref{prop_choi_infdim} to $Y=\sum_j|\lambda_j|^2|g_j\rangle\langle g_j|\otimes Z\in\mathcal B^2(\mathcal H\otimes\mathcal Z)$ for some $Z\in\mathcal B^2(\mathcal Z)\setminus\mathcal B^1(\mathcal H)$ and any $\lambda\in\ell^2(J,\mathbb C)$: The idea is that the pre-image of $Y$ under $\mathsf C_{\lambda,G}$
would be $X\mapsto{\rm tr}(X)Z$ which is not in the domain of $\mathsf C_{\lambda,G}$ as $Z$ is not trace class.
Even worse: as mentioned before, the failure of the channel-state duality also persists when
restricting $\mathsf C_{\lambda,G}$ to a map between the completely bounded maps and the trace class on $\mathcal H\otimes\mathcal Z$ \cite[Thm.~3.3]{HKS24}.

\section{(Super)operators on Schatten Classes and Their Operator-Sum Forms}\label{sec_prelim_2}

The key feature of the Choi formalism is that it allows to explicitly construct so-called \textit{Kraus operators} of any completely positive linear map, as first shown by Choi \cite{Choi75}{}.
More precisely, given 
any $\Phi\in\mathsf{CP}(\mathbb C^n,\mathbb C^k)$
the term ``Kraus operators'' refers to any finite family $\{V_j\}_{j\in J}\subset \mathbb C^{k\times n}$ such that $\Phi=\sum_{j\in J}V_j(\cdot)V_j^*$, cf.~\cite[Ch.~4.2]{Heinosaari12}.
Interestingly, this characterization is well known to carry over to infinite dimensions; however, there one has to be careful about how to interpret the---possibly uncountably infinite---sum $\sum_{j\in J}V_j(\cdot)V_j^*$.
This is where common topologies on $\mathcal B(\mathcal X,\mathcal Y)$ (with $\mathcal X,\mathcal Y$ normed spaces) weaker than the norm topology come into play, more precisely the
\textit{strong operator topology}---which is the topology induced by the seminorms $\{T\mapsto \|Tx\|\}_{x\in \mathcal X}$---and
the \textit{weak operator topology}---which is induced by the seminorms $\{T\mapsto |\varphi(Tx)|\}_{x\in \mathcal X,\varphi\in \mathcal Y^*}$ (where $\mathcal Y^*$ denotes the topological dual of $\mathcal Y$). If $\mathcal Y$ is a Hilbert space the seminorms inducing the weak operator topology reduce to $|\langle y,Tx\rangle|$. For more details we refer, e.g., to \cite[Ch.~VI.1]{Dunford58} or  \cite[Prop.~2.1.20]{vE_PhD_2020}.

The first technicality regarding the Kraus form---sometimes also called operator-sum form---in infinite dimensions is that for maps $\Phi\in\mathsf{CP}(\mathcal H,\mathcal Z)$ which are additionally trace-preserving the Kraus operators satisfy $\sum_{j\in J}V_j^*V_j={\bf1}_{\mathcal H}$.
Because we allow for non-separable spaces we choose to be explicit about how such sums are to be understood.
The statement here---which goes back to Vigier \cite{Vigier46} and which can be found in
various versions in \cite[Ch.~1.6]{Davies76}, \cite[Lemma~1.7.4]{Sak71}, \cite[Ch.~8, Thm.~3.3]{Kato80}, and \cite[Appendix~II]{Dix81}---is that norm-bounded increasing nets of self-adjoint operators converge automatically in various topologies.

\begin{lemma}\label{lemma_krausopconv}
Let complex Hilbert spaces $\mathcal H,\mathcal Z$ and
$\{V_j\}_{j\in J}\subset\mathcal B(\mathcal H,\mathcal Z)$ be given.
Assume that $\{\sum_{j\in F}V_j^*V_j\}_{F\subseteq J\text{ finite}}$ is uniformly bounded, i.e.~there exists $C> 0$ such that 
$\|\sum_{j\in F}V_j^*V_j\|_\infty\leq C$ for
all $F\subseteq J$ finite.
Then $\{\sum_{j\in F}V_j^*V_j\}_{F\subseteq J\text{ finite}}$ admits a supremum $X\in\mathcal B(\mathcal H)$, i.e.~
$\sum_{j\in F}V_j^*V_j\leq X$
for all $F\subseteq J$ finite, and if some $Y\in\mathcal B(\mathcal H)$ satisfies $\sum_{j\in F}V_j^*V_j\leq Y$ for all $F\subseteq J$ finite,
then $X\leq Y$.
Moreover, $\{\sum_{j\in F}V_j^*V_j\}_{F\subseteq J\text{ finite}}$ converges to $X$ in the strong operator topology as well as the ultraweak topology.
In particular, $X\geq 0$.
\end{lemma}
\begin{proof}
Let any $F\subseteq J$ finite be given.
Because $\sum_{j\in F}V_j^*V_j$ is self-adjoint, $\|\sum_{j\in F}V_j^*V_j\|_\infty\leq C$ is equivalent to $\sum_{j\in F}V_j^*V_j\leq C\cdot{\bf 1}$.
With this, \cite[Appendix~II]{Dix81} guarantees the existence of the supremum $X$, and the proof of said result shows that $X$ is precisely the limit of $\{\sum_{j\in F}V_j^*V_j\}_{F\subseteq J\text{ finite}}$ in the strong operator topology (as $\{\tfrac1C\sum_{j\in F}V_j^*V_j\}_{F\subseteq J\text{ finite}}\subset N$, in Dixmier's notation).
Finally, $X\geq 0$ is a direct consequence of the supremum property, and ultraweak convergence is due to \cite[Lemma~1.7.4]{Sak71}.
\end{proof}
\noindent In what follows, for uniformly bounded nets $\{\sum_{j\in F}V_j^*V_j\}_{F\subseteq J\text{ finite}}$ we denote the strong limit
by $\sum_{j\in J}V_j^*V_j$.
This condition of uniform boundedness is the key to establishing in which sense the Kraus form converges.
\begin{lemma}\label{lemma_krausmap_def}
Let
any $\{V_j\}_{j\in J}\subset\mathcal B(\mathcal H,\mathcal Z)$ be given such that
$\{\sum_{j\in F}V_j^*V_j\}_{F\subseteq J\text{ finite}}$ is uniformly bounded.
Then 
the following statements hold.
\begin{itemize}
\item[(i)] For all $B\in\mathcal B(\mathcal Z)$, $\{\sum_{j\in F}V_j^*BV_j\}_{F\subseteq J\text{ finite}}$ converges strongly and ultraweakly.
The linear map $\Phi_V^*:\mathcal B(\mathcal Z)\to\mathcal B(\mathcal H)$, $B\mapsto \sum_{j\in J}V_jBV_j^*$ is completely positive and bounded with norm $\|\Phi_V^*\|_{\infty\to\infty}=\| \sum_{j\in J}V_j^*V_j \|_\infty$.
\item[(ii)] For all $A\in\mathcal B^1(\mathcal H)$,
$\{\sum_{j\in F}V_jAV_j^*\}_{F\subseteq J\text{ finite}}$ converges in trace norm.
The linear map $\Phi_V:\mathcal B^1(\mathcal H)\to\mathcal B^1(\mathcal Z)$, $A\mapsto \sum_{j\in J}V_jAV_j^*$ is completely positive and bounded with norm $\|\Phi_V\|_{1\to 1}=\| \sum_{j\in J}V_j^*V_j \|_\infty$.
\end{itemize}
\end{lemma}
\begin{proof}
(i): 
If $B$ is self-adjoint, then $V_j^*BV_j\leq \|B\|_\infty V_j^*V_j$ for all $j\in J$:
using Cauchy-Schwarz one for all $x\in\mathcal H$ has
\begin{align*}
\langle x,(\|B\|_\infty V_j^*V_j-V_j^*BV_j)x\rangle
&\geq \|B\|_\infty\|V_jx\|^2-\|V_jx\|\|BV_jx\|\\
&\geq \|B\|_\infty\|V_jx\|^2-\|B\|_\infty\|V_jx\|^2=0\,.
\end{align*}
Thus, because $\{\sum_{j\in F}V_j^*V_j\}_{F\subseteq J\text{ finite}}$ is assumed to be uniformly bounded
the same is true for $\{\sum_{j\in F}V_j^*BV_j\}_{F\subseteq J\text{ finite}}=\{\sum_{j\in F}(\sqrt BV_j)^*(\sqrt BV_j)\}_{F\subseteq J\text{ finite}}$ and all $B\in\mathcal B(\mathcal Z)$, $B\geq 0$.
Hence Lemma~\ref{lemma_krausopconv} establishes strong
and ultraweak convergence to a bounded, positive semi-definite operator $\Phi_V^*(B)$.
The general case follows from the well-known fact that every 
bounded operator can be written as the linear combination of four positive semi-definite 
operators \cite[Coro.~4.2.4]{Kadison83}\footnote{
While the cited result deals with bounded operators, from the proof it is obvious that if the operator to be decomposed is trace class, then the
decomposing operators are trace class, as well.\label{footnote_4pos_traceclass}
}.
This defines a positive linear map $\Phi_V^*:\mathcal B(\mathcal Z)\to\mathcal B(\mathcal H)$, which by the Russo-Dye theorem \cite[Coro.~2.9]{Paulsen03}
is bounded with $\|\Phi_V^*\|_{\infty\to\infty}=\|\Phi_V^*({\bf1})\|_\infty=\|\sum_{j\in J}V_j^*V_j\|_\infty$.
Repeating this procedure for $\{\sum_{j\in F}({\bf1}_n\otimes V_j)^*B({\bf1}_n\otimes V_j)\}_{F\subseteq J\text{ finite}}$ where $B\in\mathcal B(\mathbb C^n\otimes\mathcal Z)$, $n\in\mathbb N$ are arbitrary even proves complete positivity of $\Phi_V^*$.

(ii): 
As before, one first shows summability for $A\geq 0$; given this step is a bit more technical we outsourced it to the appendix
as Lemma~\ref{lemma_strongconv_kraus_app}.
The general case follows, again, from the decomposition into four positive semi-definite operators (cf.~footnote~\ref{footnote_4pos_traceclass})
which establishes a positive---hence bounded, cf.~\cite[Lemma~2.3]{LiDu15} or \cite[Lemma~2.3.6]{vE_PhD_2020}---linear operator $\Phi_V$.
Using ultraweak convergence, we see that the map $\Phi_V^*$ from (i) is precisely the dual of $\Phi_V$ because
$
{\rm tr}(B\Phi_V(A))=\lim_F{\rm tr}(\sum_{j\in F}BV_j^*AV_j)=\lim_F{\rm tr}(\sum_{j\in F}V_jBV_j^*A)={\rm tr}(\Phi_V^*(B)A)
$
for all $A\in\mathcal B^1(\mathcal H)$, $B\in\mathcal B(\mathcal Z)$.
Hence $\Phi_V$ is completely positive because $\Phi_V^*$ is, and
$\|\Phi_V\|_{1\to 1}=\|\Phi_V^*\|_{\infty\to\infty}=\|\sum_{j\in J}V_j^*V_j\|_\infty$.
\end{proof}

\noindent With this we can now state
how complete positivity and the Kraus form are related in general, cf.~\cite[Ch.~9.2]{Davies76}, \cite[Thm.~1.4]{LiDu15}, or \cite[Prop.~2.3.10 ff.]{vE_PhD_2020}---which will be another key ingredient in the proof of our main result:
For every $\Phi\in\mathsf{CP}(\mathcal H,\mathcal Z)$ (resp.~$\Phi\in\mathcal L(\mathcal B(\mathcal H),\mathcal B(\mathcal Z))$ completely positive and ultraweakly continuous)
there exist $\{V_j\}_{j\in J}\subset\mathcal B(\mathcal H,\mathcal Z)$ such that the set $\{\sum_{j\in F}V_j^*V_j\}_{F\subseteq J\text{ finite}}$ is uniformly bounded, and one has $\Phi(A)=\sum_{j\in J}V_jAV_j^*$ for all $A\in\mathcal B^1(\mathcal H)$ (resp.~$\Phi(B)=\sum_{j\in J}V_j^*BV_j$ for all $B\in\mathcal B(\mathcal Z)$) where the respective sums converge as described in Lemma~\ref{lemma_krausmap_def}.
Moreover, if $\mathcal H,\mathcal Z$ are both separable, then one can choose the index set $J$ to be countable.
We remark that the Kraus form has also been extended to factors on separable Hilbert spaces \cite[Thm.~3.1]{HKS24}.\medskip

The final thing we have to do before coming to our main result is to go one level higher as we need to make sense of the trace of superoperators $\Phi(B^*(\cdot)B)$, respectively establish criteria under which such a trace exists.
Because $\mathcal B^2(\mathcal H)$ is a Hilbert space, the trace class $\mathcal B^1(\mathcal B^2(\mathcal H))$ on such spaces is well defined.
The following lemma is about when maps $\Phi(B^*(\cdot)B)$ are themselves trace class:
\begin{lemma}\label{lemma_weighted_Phi_welldef}
The following statements hold.
\begin{itemize}
\item[(i)] For all $p>0$ and all $X,Y\in\mathcal B^p(\mathcal H)$ one has $X^*(\cdot)Y\in\mathcal B^p(\mathcal B^2(\mathcal H))$ with $\|X^*(\cdot)Y\|_p=\|X\|_p\|Y\|_p$.
\item[(ii)] For all $\Phi\in\mathcal B(\mathcal B^1(\mathcal H))$, $B\in\mathcal B^{2}(\mathcal H)$
one has
$\Phi(B^*(\cdot)B)\in\mathcal B^2(\mathcal B^2(\mathcal H))$.
\item[(iii)] For all $\Phi\in\mathcal B(\mathcal B^1(\mathcal H))$, $B\in\mathcal B^{1}(\mathcal H)$
one has
$\Phi(B^*(\cdot)B)\in\mathcal B^1(\mathcal B^2(\mathcal H))$.
\item[(iv)] For all $K_1,K_2\in\mathcal B(\mathcal H)$, $B\in\mathcal B^{1}(\mathcal H)$ the map $X\mapsto K_1B^*XBK_2$ is in $\mathcal B^1(\mathcal B^2(\mathcal H))$, and its trace is given by ${\rm tr}(K_1B^*){\rm tr}(BK_2)$.
\end{itemize}
\end{lemma}
\begin{proof}
(i):
Given any Schmidt decompositions $X=\sum_{j\in N_1}s_j(X)|f_j\rangle\langle g_j|$, $Y=\sum_{j\in N_2}s_j(Y)|y_j\rangle\langle z_j|$,
we write $|\,  |g_j\rangle\langle z_k|\,\rangle\langle \,|f_j\rangle\langle y_k| \,|$ as short-hand for the map $B\mapsto\langle\,|f_j\rangle\langle y_k|,B\rangle_\mathrm{HS}|g_j\rangle\langle z_k|=\langle f_j|B|y_k\rangle |g_j\rangle\langle z_k|$ on $\mathcal B^2(\mathcal H)$.
This yields
$
X^*(\cdot)Y
=\sum_{(j,k)\in N_1\times N_2}s_j(X)s_k(Y)|\,  |g_j\rangle\langle z_k|\,\rangle\langle \,|f_j\rangle\langle y_k| \,|
$
which, in particular, is a Schmidt decomposition 
because $\{\, |g_j\rangle\langle z_k| \,\}_{(j,k)\in N_1\times N_2}$, $\{\, |f_j\rangle\langle y_k| \,\}_{(j,k)\in N_1\times N_2}$ are orthonormal sets in $\mathcal B^2(\mathcal H)$.
Thus the singular values of $X^*(\cdot)Y$ are $\{ s_j(X)s_k(Y) \}_{(j,k)\in N_1\times N_2}$, meaning
$
\|X^*(\cdot)Y\|_p
=(\sum_{j\in N_1}(s_j(X))^p)^{1/p}(\sum_{k\in N_2}(s_k(Y))^p)^{1/p}=\|X\|_p\|Y\|_p
$.

(ii): First we show that $\Phi(B^*(\cdot)B)$ is bounded on $\mathcal B^2(\mathcal H)$: Using Lemma~\ref{lemma_schatten_comp}
we upper bound $\|\Phi(B^*(\cdot)B)\|_{2\to 2}$
by
$
\sup_{X\in\mathcal B^2(\mathcal H),\|X\|_2\leq 1}\|\Phi(B^*XB)\|_1
\leq\|\Phi\|_{1\to 1}\sup_{X\in\mathcal B^2(\mathcal H),\|X\|_2\leq 1}\|B^*\|\|X\|_2\|B\|_2=\|\Phi\|_{1\to 1}\|B\|\|B\|_2<\infty$.
Next consider any Schmidt decomposition $B=\sum_{j\in N}s_j(B)|f_j\rangle\langle g_j|$. Completing the orthonormal system
$\{f_j\}_{j\in N}$ to an orthonormal basis 
$\{f_j\}_{j\in J}$ of $\mathcal H$
gives rise to an orthonormal basis
$\{|f_k\rangle\langle f_j|\}_{j,k\in J}$ of $\mathcal B^2(\mathcal H)$.
By Theorem~2.4.3 in \cite{Ringrose71}{}, $\Phi(B^*(\cdot)B)\in\mathcal B^2(\mathcal B^2(\mathcal H))$ is now equivalent to $\sum_{j,k\in J}\|\Phi(B^*|f_k\rangle\langle f_j|B)\|_2^2<\infty$.
But the latter sum evaluates to
\begin{align*}
\sum_{j,k\in J}\|\Phi(B^*|f_k\rangle\langle f_j|B)\|_2^2&=\sum_{j,k\in N}(s_j(B))^2(s_k(B))^2
\|\Phi(|g_k\rangle\langle g_j|)\|_2^2\\
&\leq\|\Phi\|_{1\to 1}^2\Big(\sum_{j\in N}(s_j(B))^2\Big)^2=\|\Phi\|_{1\to 1}^2\|B\|_2^4<\infty
\end{align*}
due to $\|\cdot\|_2\leq\|\cdot\|_1$. 
(iii): By Lemma~\ref{lemma_schatten_comp}~(ii) there exist $B_1,B_2\in\mathcal B^2(\mathcal H)$ such that $B=B_1B_2$,
so we decompose $\Phi(B^*(\cdot)B)=\Phi((B_1B_2)^*(\cdot)B_1B_2)=\Phi(B_2^*(\cdot)B_2)\circ(B_1^*(\cdot)B_1)$.
By (i), $B_1^*(\cdot)B_1\in\mathcal B^2(\mathcal B^2(\mathcal H))$, and the same holds for $\Phi(B_2^*(\cdot)B_2)$ by (ii).
In particular, Lemma~\ref{lemma_schatten_comp}~(i) implies that
their composition is in $\mathcal B^1(\mathcal B^2(\mathcal H))$.

(iv): 
As $BK_1^*,BK_2\in\mathcal B^1(\mathcal H)$ we know $K_1B^*(\cdot)BK_2\in\mathcal B^1(\mathcal B^2(\mathcal H))$ by (i).
Given any orthonormal basis $\{g_j\}_{j\in J}$ of $\mathcal H$ (so $\{|g_k\rangle\langle g_j|\}_{j,k\in J}$ is an orthonormal basis of $\mathcal B^2(\mathcal H)$) we can evaluate the trace in question explicitly:
\begin{align*}
{\rm tr}(K_1B^*(\cdot)BK_2)&=\sum_{j,k\in J}\big\langle\,|g_k\rangle\langle g_j|  ,  K_1B^*|g_k\rangle\langle g_j|BK_2 \big\rangle_{\sf HS}\\
&=\sum_{j,k\in J}\langle g_k,K_1B^*g_k\rangle\langle g_j,BK_2g_j\rangle={\rm tr}(K_1B^*){\rm tr}(BK_2) 
\end{align*}
\end{proof}

A fact that will become important later is that given $\Phi$ completely positive, the trace of $\Phi(B^*(\cdot)B)$ can be evaluated explicitly using the Kraus operators of $\Phi$.
\begin{lemma}\label{lemma_v_ell2}
Let 
$B\in\mathcal B^{1}(\mathcal H)$
as well as $\Phi\in\mathsf{CP}(\mathcal H)$ be given. If $\{V_j\}_{j\in J}\subset\mathcal B(\mathcal H)
$ is any set of Kraus operators of $\Phi$, then $v:=\{{\rm tr}(B^*V_j)\}_{j\in J}\in\ell^2(J,\mathbb C)$.
Moreover, the trace of $\Phi(B^*(\cdot)B)$ is equal to $\|v\|_2^2$.
\end{lemma}
\begin{proof}
As done before we decompose $B$ into Hilbert-Schmidt operators $B_1:=\sum_{j\in N}\sqrt{s_j}|f_j\rangle\langle g_j|$, $B_2:=\sum_{j\in N}\sqrt{s_j}|g_j\rangle\langle g_j|$, that is, $B=B_1B_2$.
We compute
$
\sum_{j\in F}|{\rm tr}(B^*V_j)|^2=\sum_{j\in F}|\langle B_1 , V_jB_2^* \rangle_{\sf HS}|^2
\leq\|B_1\|_2^2\sum_{j\in F}{\rm tr}(B_2^*B_2V_j^*V_j)
$
where in the second step we used the Cauchy-Schwarz inequality on $\mathcal B^2(\mathcal H)$
Now by Lemma~\ref{lemma_krausopconv} $\{\sum_{j\in F}V_j^*V_j\}_{F\subseteq J\text{ finite}}$
converges ultraweakly to its supremum so
\begin{align}
\sum_{j\in F}|{\rm tr}(B^*V_j)|^2&\leq\|B_1\|_2^2\sum_{j\in F}{\rm tr}(B_2^*B_2V_j^*V_j)\leq\|B_1\|_2^2\Big\|B_2^*B_2\sum_{j\in F}V_j^*V_j\Big\|_1\notag\\
&\leq \|B_1\|_2^2\|B_2^*B_2\|_1\Big\|\sum_{j\in F}V_j^*V_j\Big\|_\infty\leq \|B\|_1^2\Big\|\sum_{j\in J}V_j^*V_j\Big\|_\infty.\label{eq:lemma_v_ell2-1}
\end{align}
This implies two things:
1.~$v=\{{\rm tr}(B^*V_j)\}_{j\in J}\in\ell^2(J,\mathbb C)$---as desired---as well as
2.~%
$w:=\{{\rm tr}(B_2V_j^*V_jB_2^*)\}_{j\in J}=\{\|B_2^*V_j\|_2^2\}_{j\in J}\in\ell^1(J,\mathbb C)$.
To connect $\|v\|_2^2$ to the trace of $\Phi(B^*(\cdot)B)$
we complete $\{f_j\}_{j\in N}$ to an orthonormal basis $\{f_a\}_{a\in A}$ of $\mathcal H$ (i.e.~$N\subseteq A$), which lets us evaluate
\begin{align*}
{\rm tr}(\Phi(B^*(\cdot)B))&=\sum_{a,b\in A}\big\langle\,|f_a\rangle\langle f_b|, \Phi\big( B^*|f_a\rangle\langle f_b|B\big)\big\rangle_\mathrm{HS}\\
&=\sum_{a,b\in N}\sum_{j\in J}\sqrt{s_as_b}\langle f_a,V_jB_2^*g_a\rangle\langle g_b,B_2V_j^*f_b\rangle\,.
\end{align*}
Now we want to interchange the order of summation; for this we check that $\{\sqrt{s_as_b}\langle f_a,V_jB_2^*g_a\rangle\langle g_b,B_2V_j^*f_b\rangle\}_{a,b\in N,j\in J}$ is
summable \cite[Lemma~1.2.6]{Ringrose71}.
Using Cauchy-Schwarz on $\ell^2(N,\mathbb C)$ we compute
\begin{align*}
\sum_{j\in J}\sum_{a,b\in N}\big|\sqrt{s_as_b}\langle f_a&,V_jB_2^*g_a\rangle\langle g_b,B_2V_j^*f_b\rangle\big|\\
&=\sum_{j\in J}\Big(
\sum_{a\in N}\sqrt{s_a}|\langle f_a,V_jB_2^*g_a\rangle|\Big)^2\\
&\leq \sum_{j\in J}\Big(\sum_{a\in N}(\sqrt{s_a})^2\Big)  \Big( \sum_{a\in N}|\langle f_a,V_jB_2^*g_a\rangle|^2 \Big)\\
&\leq \Big(\sum_{a\in N}s_a\Big)\sum_{j\in J}\Big( \sum_{a\in A}\|V_jB_2^*g_a\|^2 \Big)\\
&=\|B\|_1\sum_{j\in J}{\rm tr}(B_2V_j^*V_jB_2^*)
=\|B\|_1\|w\|_1
<\infty\,,
\end{align*}
which by \cite[Lemma~1.2.5]{Ringrose71} implies summability, as desired.
Thus we may interchange sums which lets us arrive at
\begin{align*}
{\rm tr}(\Phi(B^*(\cdot)B))&=\sum_{j\in J}\sum_{a,b\in N}\sqrt{s_as_b}\langle f_a,V_jB_2^*g_a\rangle\langle g_b,B_2V_j^*f_b\rangle\\
&=\sum_{j\in J}\sum_{a,b\in J}\langle f_a,V_jB_2^*B_1^*f_a\rangle\langle f_b,B_1B_2V_j^*f_b\rangle\\
&=\sum_{j\in J}
|{\rm tr}(V_j(B_1B_2)^*)|^2
=\sum_{j\in J}|{\rm tr}(B^*V_j)|^2=\|v\|_2^2\,.
\end{align*}
\end{proof}

\section{Unique Decompositions of Generators of Completely Positive Dynamical Semigroups}\label{sec_uniquedecomp}

Finally, we are ready to establish our main result. For this we first single out the subspace of $\mathsf{CP}(\mathcal H)$ for which the weighted trace 
from Lemma~\ref{lemma_weighted_Phi_welldef}~(iii)
vanishes.

\begin{definition}\label{lemma_trace_infdim}
Given $B\in\mathcal B^{1}(\mathcal H)$ where $\mathcal H$ is an arbitrary complex Hilbert space. Then
$\mathsf{CP}_B(\mathcal H):=\{\Phi\in \mathsf{CP}(\mathcal H)\,:\,{\rm tr}(\Phi(B^*(\cdot)B))=0\}$.
\end{definition}
\noindent By Lemma~\ref{lemma_v_ell2} $\Phi\in \mathsf{CP}_B(\mathcal H)$ if and only if there exist Kraus operators $\{V_j\}_{j\in J}$ of $\Phi$ such that ${\rm tr}(B^*V_j)=0$ for all $j\in J$ (equivalently: all sets of Kraus operators satisfy this).
More importantly, this trace can also be recovered using the Choi formalism and vectorization. Let us quickly recap the latter concept:
In finite dimensions, given $B\in\mathbb C^{m\times n}$ and any orthonormal basis
$G:=\{g_j\}_{j=1}^{n}$ of 
$\mathbb C^n$
one defines
$
\vec{G}B:=
\sum_{k=1}^m 
 g_j\otimes Bg_j\in\mathbb C^n\otimes\mathbb C^m\simeq\mathbb C^{mn}
$
\cite[Ch.~4.2 ff.]{HJ2}.
This generalizes to any complex Hilbert spaces $\mathcal H,\mathcal Z$ 
via the map
$
{\rm vec}_{G}:\mathcal B^2(\mathcal H,\mathcal Z)
\to \mathcal H\otimes\mathcal Z
,
X
\mapsto \sum_{j\in J}g_j\otimes Xg_j 
$
(which is well-defined and unitary)
where $G:=\{g_j\}_{j\in J}$ is any orthonormal basis of $\mathcal H$, refer, e.g., to \cite{Gudder20}{}.
Equivalently, given any $X\in\mathcal B^2(\mathcal H)$, $\vec{G}X$ is the unique element of $\mathcal H\otimes\mathcal Z$ such that $\langle g_j\otimes z,\vec{G}X\rangle=\langle z,Xg_j\rangle $ for all $j\in J$, $z\in\mathcal Z$.
\begin{lemma}\label{lemma0_infdim}
Given any orthonormal basis $G:=\{g_j\}_{j\in J}$ of a complex Hilbert space $\mathcal H$ and any $\lambda\in\ell^2(J,\mathbb C)$ for all $\Phi\in\mathcal B(\mathcal B^1(\mathcal H))$ and all $X,Y\in\mathcal B^2(\mathcal H)$ it holds that
$
{\rm tr}\big(\Phi((XB)^*(\cdot)YB)\big)=\langle \vec{G}{X},\mathsf C_{\lambda,G}(\Phi)\vec{G}Y\rangle
$,
where $B:=\sum_{j\in J}\lambda_j|g_j\rangle\langle g_j|$.
\end{lemma}
\begin{proof}
We begin by expanding the right-hand side:
\begin{align*}
\langle& \vec{G}{X},\mathsf C_{\lambda,G}(\Phi)\vec{G}Y\rangle\\
&\ \ =\sum_{p,j,k,q\in J}\lambda_j^*\lambda_k\big\langle  g_p\otimes Xg_p,(  |g_j\rangle\langle g_k|\otimes\Phi(|g_j\rangle\langle g_k|)  )(  g_q\otimes Yg_q  )     \big\rangle\\
&\ \ =\sum_{j,k\in J}\lambda_j^*\lambda_k\big\langle   Xg_j,\Phi(|g_j\rangle\langle g_k|) Yg_k \big\rangle=\sum_{j,k\in J}\big\langle   g_j,X^*\Phi(B^*|g_j\rangle\langle g_k|B) Yg_k \big\rangle\,.
\end{align*}
Defining $\Phi_{X^*,Y}:=X^*(\cdot)Y$, $\Phi_{B^*,B}:=B^*(\cdot)B$---which are both elements of $\mathcal B^2(\mathcal B^2(\mathcal H))$ by Lemma~\ref{lemma_weighted_Phi_welldef}~(i)---the map $\Phi_{X^*,Y}\circ\Phi\circ \Phi_{B^*,B}=X^*\Phi(B^*(\cdot)B)Y$ is in $\mathcal B^1(\mathcal B^2(\mathcal H))$ (Lemma~\ref{lemma_schatten_comp}~(i) and Lemma~\ref{lemma_weighted_Phi_welldef}~(ii)).
Because $\{|g_k\rangle\langle g_j|\}_{j\in J,k\in K}$ is an orthonormal basis
$\sum_{j,k\in J}\langle   g_j,X^*\Phi(B^*|g_j\rangle\langle g_k|B) Yg_k \rangle$ is 
precisely the trace of $\Phi_{X^*,Y}\circ\Phi\circ \Phi_{B^*,B}$.
Hence
$
\langle \vec{G}{X}|\mathsf C_{\lambda,G}(\Phi)|\vec{G}Y\rangle={\rm tr}(\Phi_{X^*,Y}\circ\Phi\circ \Phi_{B^*,B})
={\rm tr}(\Phi\circ \Phi_{B^*,B}\circ \Phi_{X^*,Y})={\rm tr}\big(\Phi((XB)^*(\cdot)YB)\big)$,
as claimed.
\end{proof}

Because the Choi formalism turns completely positive inputs into positive semi-definite operators, the above identity places a restriction on the kernel of the Choi operator in case $\Phi$ is completely positive and has vanishing weighted trace:

\begin{lemma}\label{lemma_vec_kernel}
Let any orthonormal basis $G:=\{g_j\}_{j\in J}$ of a complex Hilbert space $\mathcal H$ and any $\lambda\in\ell^2(J,\mathbb C)$ be given. Then for all $X\in\mathcal B^2(\mathcal H)$ and all $\Phi\in\mathsf{CP}_{XB}(\mathcal H)$---where
$B:=\sum_{j\in J}\lambda_j|g_j\rangle\langle g_j|\in\mathcal B^2(\mathcal H)$---it holds that $\mathsf C_{\lambda,G}(\Phi)\vec{G}X=0$
\end{lemma}
\begin{proof}
Because $\Phi$ is completely positive,
$\mathsf C_{\lambda,G}(\Phi)$
is positive semi-definite by Lemma~\ref{lemma_choi}.
On the other hand, ${\rm tr}(\Phi((XB)^*(\cdot)XB))=0$
by Lemma~\ref{lemma0_infdim} is equivalent to $\langle \vec{G}{X},\mathsf C_{\lambda,G}(\Phi)\vec{G}X\rangle=0$.
Therefore
\begin{align*}
0&=\langle \vec{G}{X},\mathsf C_{\lambda,G}(\Phi)\vec{G}X\rangle\\
&=\Big\langle \sqrt{\mathsf C_{\lambda,G}(\Phi)}\vec{G}{X},\sqrt{\mathsf C_{\lambda,G}(\Phi)}\vec{G}X\Big\rangle=\Big\| \sqrt{\mathsf C_{\lambda,G}(\Phi)}\vec{G}{X}\Big\|^2\,.
\end{align*}
Altogether this shows $\mathsf C_{\lambda,G}(\Phi)\vec{G}{X} =\sqrt{\mathsf C_{\lambda,G}(\Phi)}\sqrt{\mathsf C_{\lambda,G}(\Phi)}\vec{G}{X}=0$.
\end{proof}
It turns out that the kernel constraint from Lemma~\ref{lemma_vec_kernel} is the reason why the $K(\cdot)+(\cdot)K^*$-part of any generator $L\in\mathsf L(\mathsf{CP}(\mathcal H))$ is ``independent'' of the completely positive part, assuming the weighted trace of the latter vanishes.
This is also the point where we have to invoke separability of the underlying Hilbert space:
\begin{proposition}\label{prop_CPB_diff_impossible}
Let $K\in\mathcal B(\mathcal H)$ where $\mathcal H$ is a separable complex Hilbert space, as well as $B\in\mathcal B^{1}(\mathcal H)$ 
with ${\rm tr}(B)\neq 0$ be given.
If 
$K(\cdot)+(\cdot)K^*\in\mathsf{CP}_B(\mathcal H)-\mathsf{CP}_B(\mathcal H)$, then
$K=i\lambda{\bf1}$ for some $\lambda\in\mathbb R$.
\end{proposition}
\begin{proof}
We start with any Schmidt decomposition $B=\sum_{j\in N}s_j(B)|f_j\rangle\langle g_j|$.
If $\{g_j\}_{j\in N}$ is not yet an orthonormal basis 
we can extend it to an orthonormal basis $G:=\{g_j\}_{j\in J}$, $N\subseteq J$ of $\mathcal H$;
in particular $J$---and hence $J\setminus N$---are countable by assumption.
Define
$s_j(B):=0$ for $j\in J\setminus N$ as well as 
$B_1:=\sum_{j\in J}\sqrt{s_j(B)}|f_j\rangle\langle g_j|$, $B_2:=\sum_{j\in J}\lambda_j|g_j\rangle\langle g_j|$
where
$\lambda:J\to(0,\infty)$ is defined via
$\lambda(j):=\sqrt{s_j(B)}$ if $\sqrt{s_j(B)}\neq 0$, and $\lambda(j):=2^{-j}$ otherwise.
Because $J$ is countable,
$\lambda$ is square summable
meaning 
$B_1,B_2\in\mathcal B^2(\mathcal H)$ and
$B=B_1B_2$.
Now the key insight is that if $K(\cdot)+(\cdot)K^*=\Phi_1-\Phi_2$ for some
$\Phi_1,\Phi_2\in \mathsf{CP}_B(\mathcal H)=\mathsf{CP}_{B_1B_2}(\mathcal H)$,
then
\begin{align*}
\mathsf C_{\lambda,G}(K(\cdot)+(\cdot)K^*)\vec{G}{B_1}&=\mathsf C_{\lambda,G}(\Phi_1-\Phi_2)\vec{G}{B_1}\\
&=\mathsf C_{\lambda,G}(\Phi_1)\vec{G}{B_1}-\mathsf C_{\lambda,G}(\Phi_2)\vec{G}{B_1}=0
\end{align*}
by Lemma~\ref{lemma_choi} \& Lemma~\ref{lemma_vec_kernel}.
In particular, for all $j,k\in J$ it holds that
\begin{align*}
0&=\langle g_k\otimes g_j,\mathsf C_{\lambda,G}(K(\cdot)+(\cdot)K^*)\vec{G}{B_1}\rangle\\
&=\sum_{a,b\in J}\lambda_a\lambda_b\big\langle g_k\otimes g_j,\big(|g_a\rangle\langle g_b|\otimes(K|g_a\rangle\langle g_b|+|g_a\rangle\langle g_b|K^*)\big)\vec{G}{B_1}\big\rangle\\
&=\sum_{b\in J}\lambda_k\lambda_b \big(\langle  g_j ,  Kg_k\rangle\langle g_b,B_1g_b  \rangle+
\delta_{jk}\langle g_b,K^*   B_1g_b \rangle\big)\,.
\end{align*}
Now we use that
$\lambda_bB_1g_b=B_1B_2g_b=Bg_b$ for all $b\in J$
to compute
\begin{align}
0&=\sum_{b\in J}\lambda_k \big(\langle  g_j ,  Kg_k\rangle\langle g_b,Bg_b  \rangle+
\delta_{jk}\langle g_b,K^*   Bg_b \rangle\big)\notag\\
&= \lambda_k\langle  g_j ,  Kg_k\rangle \sum_{b\in J}\langle g_b,Bg_b\rangle+ \lambda_k\delta_{jk} \sum_{b\in J}\langle g_b,K^*Bg_b\rangle\notag\\
&=\lambda_k\langle  g_j ,  Kg_k\rangle{\rm tr}(B)+\lambda_k\delta_{jk}{\rm tr}(BK^*)\,.\label{eq:prop_CPB_diff_impossible_1}
\end{align}
Because $\lambda_k\neq 0$ for all $k\in J$ by construction,~\eqref{eq:prop_CPB_diff_impossible_1} for all $j,k\in J$ is equivalent to $0=\langle  g_j ,  Kg_k\rangle{\rm tr}(B)+\delta_{jk}{\rm tr}(BK^*)$.
This lets us distinguish two cases: if $j\neq k$, then this computation shows $0=  \langle  g_j ,  Kg_k\rangle{\rm tr}(B)$; 
but 
${\rm tr}(B)\neq 0$ by assumption, so this even shows $ \langle  g_j ,  Kg_k\rangle=0$ for all $j\neq k$.
Because $\{g_j\}_{j\in J}$ is an orthonormal basis of $\mathcal H$ this means that $K$ is diagonal (w.r.t.~$G$).
Now if $j=k$, then~\eqref{eq:prop_CPB_diff_impossible_1} yields
$
0=  \langle  g_j ,  Kg_j\rangle{\rm tr}(B)+ {\rm tr}(BK^*)
$. Again, 
${\rm tr}(B)\neq 0$ so we find
$
 \langle  g_j ,  Kg_j\rangle=-\frac{{\rm tr}(BK^*)}{{\rm tr}(B)}
$;
in other words there exists $c\in\mathbb C$ such that $ \langle  g_j ,  Kg_j\rangle=c$ for all $j\in J$. 
Altogether, as $\{g_j\}_{j\in J}$ is an orthonormal basis this shows $K=c\cdot{\bf1}$.
All that is left to prove is that $c=i\lambda$ for some $\lambda\in\mathbb R$.
Because ${\rm tr}(\Phi_1(B^*(\cdot)B))={\rm tr}(\Phi_2(B^*(\cdot)B))=0$, the same holds for $K(\cdot)+(\cdot)K^*$.
By Lemma~\ref{lemma_weighted_Phi_welldef}~(iv) this means
$
0={\rm tr}(KB^*){\rm tr}(B)+{\rm tr}(B^*){\rm tr}(BK^*)=2{\rm Re}( {\rm tr}(KB^*){\rm tr}(B) )
$.
Now we insert $K=c\cdot{\bf 1}$ to find
$
0=2{\rm Re}(c| {\rm tr}(B)|^2 )=2| {\rm tr}(B)|^2{\rm Re}(c)
$,
and as ${\rm tr}(B)\neq 0$ this shows ${\rm Re}(c)=0$, i.e.~$c=i\lambda$ for some $\lambda\in\mathbb R$.
\end{proof}

We now come to our main result.
It turns out that the most convenient way for establishing uniqueness of decompositions $L\mapsto (K,\Phi)$ of corresponding generators
is via bijectivity of a certain map (with appropriate domain and co-domain).
This is most easily formulated in terms of the Lie wedge which is a generalization of Lie algebras central to Lie semigroup theory, cf.~also \cite{HHL89}{}:
Given a Banach space $\mathcal X$ and a norm-closed (sub-)semigroup $S$ of $\mathcal B(\mathcal X)$ which contains the identity, the \textit{Lie wedge} of $S$ is defined as $\mathsf L(S):=\{A\in\mathcal B(\mathcal X):e^{tA}\in S\text{ for all }t\geq 0\}$.
As an example, 
the standard forms of generators from the introduction
in
this language 
reads
\begin{align*}
\mathsf L(\mathsf{CPTP}(\mathcal H))&=\{-i[H,\cdot]+\Phi-\tfrac12\{\Phi^*({\bf1}),\cdot\}:H\in i\mathfrak u(\mathcal H),\Phi\in\mathsf{CP}(\mathcal H)\}\\
\mathsf L(\mathsf{CP}(\mathcal H))&=\{K(\cdot)+(\cdot)K^*+\Phi:K\in \mathcal B(\mathcal H),\Phi\in\mathsf{CP}(\mathcal H)\}
\end{align*}
where
$\mathfrak u(\mathcal H)$ is the unitary algebra, i.e.~the collection of all skew-adjoint bounded operators on $\mathcal H$.

\begin{theorem}\label{thm_main}
Let a
separable complex Hilbert space $\mathcal H$
as well as $B\in\mathcal B^{1}(\mathcal H)$ 
with ${\rm Re}({\rm tr}(B))\neq 0$ be given. Then the following map is bijective:
\begin{align*}
\Xi_B:\{K\in\mathcal B(\mathcal H):{\rm Im}({\rm tr}(B^*K))=0\}\times\mathsf{CP}_B(\mathcal H)&\to \mathsf L(\mathsf{CP}(\mathcal H))\\
(K,\Phi)&\mapsto K(\cdot)+(\cdot)K^*+\Phi
\end{align*}
\end{theorem}
\begin{proof}
First we prove surjectivity. Starting from 
any 
$L\in \mathsf L(\mathsf{CP}(\mathcal H))$
one finds $K_0\in\mathcal B(\mathcal H)$, $\Phi_0\in\mathsf{CP}(\mathcal H)$ such that $L=K_0(\cdot)+(\cdot)K_0^*+\Phi_0$.
Moreover, because $\Phi_0$ is completely positive and $\mathcal H$ is separable, as seen in Section~\ref{sec_prelim_2}
there exists $\{V_j\}_{j\in N}\subset\mathcal B(\mathcal H)$, $N\subseteq\mathbb N$ such that $\Phi_0=\sum_{j\in N}V_j(\cdot)V_j^*$
and $\sum_{j\in N}V_j^*V_j$ is uniformly bounded.
Then $v:=\{{\rm tr}(B^*V_j)\}_{j\in N}\in\ell^2( N,\mathbb C)$ (Lemma~\ref{lemma_v_ell2})
so $\sum_{j\in N}(\frac{{\rm tr}(B^*V_j)}{{\rm tr}(B^*)})^*V_j$ converges strongly:
for all $x\in\mathcal H\setminus\{0\}$
\begin{align*}
\sum_{j\in N}\big\|\tfrac{{\rm tr}(B^*V_j)}{{\rm tr}(B^*)}^*V_jx\big\|&=
|{\rm tr}(B^*)|^{-1}\sum_{j\in N}\big|{\rm tr}(B^*V_j)\big|\,\|V_jx\|\\
&\leq |{\rm tr}(B^*)|^{-1}\Big(\sum_{j\in N}\big|{\rm tr}(B^*V_j)\big|^2\Big)^{1/2}\Big(\sum_{j\in N}\big\|V_jx\big\|^2\Big)^{1/2}\\
&\leq|{\rm tr}(B^*)|^{-1}\|v\|_2\Big\| \sum_{j\in N}V_j^*V_j \Big\|_\infty^{1/2}\|x\| \,.
\end{align*}
Thus $\|\sum_{j\in N}(\frac{{\rm tr}(B^*V_j)}{{\rm tr}(B^*)})^*V_j\|_\infty\leq |{\rm tr}(B^*)|^{-1}\|v\|_2\| \sum_{j\in N}V_j^*V_j \|_\infty^{1/2}$ which together with \cite[Lemma~1.2.5]{Ringrose71} yields convergence of $\sum_{j\in N}\tfrac{{\rm tr}(B^*V_j)}{{\rm tr}(B^*)}^*V_jx$.
Altogether, when defining $\tilde V_j:=V_j -\frac{{\rm tr}(B^*V_j)}{{\rm tr}(B^*)}{\bf 1} \in\mathcal B(\mathcal H)$ for all $j\in N$
this shows that $\sum_{j\in N}\tilde V_j^*\tilde V_j$ is uniformly bounded:
for all $F\subseteq N$ finite
\begin{align*}
\Big\|\sum_{j\in F}&\tilde V_j^*\tilde V_j\Big\|_\infty=\Big\|\sum_{j\in F} \Big(V_j -\frac{{\rm tr}(B^*V_j)}{{\rm tr}(B^*)}{\bf 1}\Big)^*\Big(V_j -\frac{{\rm tr}(B^*V_j)}{{\rm tr}(B^*)}{\bf 1}\Big)\Big\|_\infty\\
&\leq \Big\| \sum_{j\in F} V_j^* V_j   \Big\|_\infty+2\Big\| \sum_{j\in F}\Big(\frac{{\rm tr}(B^*V_j)}{{\rm tr}(B^*)}\Big)^*V_j   \Big\|_\infty+ \frac{\sum_{j\in F}|{\rm tr}(B^*V_j)|^2}{|{\rm tr}(B)|^2}   \\
&\leq \Big\| \sum_{j\in N} V_j^* V_j   \Big\|_\infty+2|{\rm tr}(B^*)|^{-1}\|v\|_2\Big\| \sum_{j\in N} V_j^* V_j   \Big\|_\infty^{1/2}+ \frac{\sum_{j\in F}|{\rm tr}(B^*V_j)|^2}{|{\rm tr}(B)|^2}   
\end{align*}
This by Lemma~\ref{lemma_krausmap_def}~(ii) guarantees that 
$\Phi:=\sum_{j\in N}\tilde V_j(\cdot)\tilde V_j^*\in\mathsf{CP}(\mathcal H)$ is well defined.
Moreover, a straightforward computation shows $K_0(\cdot)+(\cdot)K_0^*+\Phi_0=\tilde K(\cdot)+(\cdot)\tilde K^*+\Phi$
when defining
$$
\tilde K:=K_0+\sum_{j\in N}\Big(\frac{{\rm tr}(B^*V_j)}{{\rm tr}(B^*)}\Big)^*V_j-\frac{\sum_{j\in N}|{\rm tr}(B^*V_j)|^2}{2|{\rm tr}(B)|^2}{\bf1}\,,
$$
cf.~also \cite[Eq.~(1.4)]{Davies80unique}.
By definition, ${\rm tr}(B^*\tilde V_j)=0$ for all $j$ so  Lemma~\ref{lemma_v_ell2} shows $\Phi\in\mathsf{CP}_B(\mathcal H)$.
All that is left is to ``shift'' $\tilde K$ such that ${\rm Im}({\rm tr}(B^* \tilde K))=0$.
Obviously, replacing $\tilde K$ by $\tilde K+i\lambda{\bf 1}$, $\lambda\in\mathbb R$ does not change $\tilde K(\cdot)+(\cdot)\tilde K^*$ and thus does not change $L$.
Moreover,
$
0={\rm Im}({\rm tr}(B^* (\tilde K+i\lambda{\bf 1})))={\rm Im}({\rm tr}(B^* \tilde K))+\lambda{\rm Re}({\rm tr}(B^*))
$
so setting $K:=\tilde K-i\frac{{\rm Im}({\rm tr}(B^* \tilde K))}{{\rm Re}({\rm tr}(B))}{\bf 1}\in\mathcal B(\mathcal H)$ yields $(K,\Phi)\in{\rm dom}(\Xi_B)$ such that $\Xi_B(K,\Phi)=
\tilde K(\cdot)+(\cdot)\tilde K^*+\Phi=K_0(\cdot)+(\cdot)K_0^*+\Phi_0=L$, as desired.

For injectivity, assume
$K_1(\cdot)+(\cdot)K_1^*+\Phi_1=K_2(\cdot)+(\cdot)K_2^*+\Phi_2$
for some $K_1,K_2\in\mathcal B(\mathcal H)$ , $\Phi_1,\Phi_2\in\mathsf{CP}_B(\mathcal H)$ such that ${\rm Im}({\rm tr}(B^* K_j))=0$, $j=1,2$; equivalently,
$
(K_2-K_1)(\cdot)+(\cdot)(K_2-K_1)^*=\Phi_1-\Phi_2\in \mathsf{CP}_B(\mathcal H)-\mathsf{CP}_B(\mathcal H)
$.
Thus Proposition~\ref{prop_CPB_diff_impossible} shows that $K_1=K_2+i\lambda{\bf1}$ for some $\lambda\in\mathbb R$.
This has two consequences: On the one hand, 
this imaginary difference between $K_1$ and $K_2$ cancels in the sense that
$K_1(\cdot)+(\cdot)K_1^*=K_2(\cdot)+(\cdot)K_2^*$
which in turn implies $\Phi_1=\Phi_2$. 
On the other hand, $K_1=K_2+i\lambda{\bf1}$ together with the trace condition on $K_1,K_2$ yields
\begin{align*}
0={\rm Im}({\rm tr}(B^* K_1))
={\rm Im}({\rm tr}(B^* K_2))+\lambda{\rm Re}({\rm tr}(B^*))=\lambda{\rm Re}({\rm tr}(B))\,.
\end{align*}
But ${\rm Re}({\rm tr}(B))\neq 0$ by assumption so $\lambda$ has to vanish, meaning $K_1=K_2$.
This concludes the proof.
\end{proof}
Because $\mathsf L(\mathsf{CPTP}(\mathcal H))\subseteq\mathsf L(\mathsf{CP}(\mathcal H))$, as a direct consequence of Theorem~\ref{thm_main} we obtain unique decompositions of generators of quantum-dynamical semigroups: this follows at once from the well-known identification $K=-\frac12\Phi^*({\bf1})-iH$.
\begin{corollary}\label{coro_unique_decomp_qds}
Let a separable complex Hilbert space $\mathcal H$
as well as $B\in\mathcal B^{1}(\mathcal H)$ 
with ${\rm Re}({\rm tr}(B))\neq 0$ be given.
Then
\begin{align*}
\hat\Xi_B:(H,\Phi)&\mapsto -i[H,\cdot]+\Phi-\Big\{\frac{\Phi^*({\bf1})}{2},\cdot\Big\}\in\mathsf L(\mathsf{CPTP}(\mathcal H))
\end{align*}
with domain $\{(H,\Phi)\in i\mathfrak u(\mathcal H)\times \mathsf{CP}_B(\mathcal H):\substack{{\rm Im}({\rm tr}(\Phi(B)))=2\,{\rm Re}({\rm tr}(B^* H))}\}
$
is bijective.
In particular, if $B\in\mathcal B^{1}(\mathcal H)$ is 
self-adjoint with ${\rm tr}(B)\neq 0$, then
\begin{align*}
\hat\Xi_B:\{H\in i\mathfrak u(\mathcal H):{\rm tr}(BH)=0\}\times \mathsf{CP}_B(\mathcal H)&\to \mathsf L(\mathsf{CPTP}(\mathcal H))\\
(H,\Phi)&\mapsto -i[H,\cdot]+\Phi-\Big\{\frac{\Phi^*({\bf1})}{2},\cdot\Big\}
\end{align*}
is bijective.
\end{corollary}
\noindent From a physics perspective this corollary shows that any fixed reference state (more precisely: reference operator) gives rise to a unique splitting of GKSL-generators.

Finally a comment on how this result translates to the dual picture. There the object of interest is
the collection of all $\Phi:\mathcal B(\mathcal H)\to\mathcal B(\mathcal Z)$ which are completely positive, unital (i.e.~$\Phi({\bf1}_{\mathcal H})={\bf1}_{\mathcal Z}$), and ultraweakly continuous, denoted by $\mathsf{CPU}(\mathcal H)_\sigma$.
It is well known (and easy to verify) that $\Phi\in\mathsf{CPU}(\mathcal H)_\sigma$ if and only if its dual map $\Phi^*$ is CPTP.
Therefore $L\in \mathsf L(\mathsf{CPTP}(\mathcal H))$ is equivalent to $L^*\in\mathsf L(\mathsf{CPU}(\mathcal H)_\sigma)$
meaning the unique decomposition $(H,\Phi)$ of $L$ from Corollary~\ref{coro_unique_decomp_qds} readily translates into a unique decomposition $(-H,\Phi^*)$ of the dual generator $L^*$.
\section{Open Questions}\label{sec_open_q}

An obvious question 
is whether separability of the underlying Hilbert space is a necessary assumption in Theorem~\ref{thm_main}.
The only point where this assumption was needed was Proposition~\ref{prop_CPB_diff_impossible} for which injectivity of the weighted Choi formalism (which only holds in the separable case, cf.~Proposition~\ref{prop_choi_infdim}) was crucial.
This does of course \textit{not} mean that our main result is wrong for non-separable Hilbert spaces, just that if it is true, then one needs an entirely different proof strategy there.

Also one may wonder whether the reason the one-to-one relation between completely positive maps and positive semi-definite operators fails in infinite dimensions (Prop.~\ref{prop_choi_infdim})
is that the Choi formalism was modified in an ``unfavorable'' way.
Indeed,
the ``unweighted'' Choi map $\Phi\mapsto\sum_{j,k\in J}|g_j\rangle\langle g_k|\otimes\Phi(|g_j\rangle\langle g_k|)$ establishes a correspondence between maps $:\mathcal B^1(\mathcal H)\to\mathcal B^1(\mathcal Z)$ \textit{of finite rank} (i.e.~maps of the form $X\mapsto\sum_{j=1}^m{\rm tr}(A_jX)B_j$ with $m<\infty$) and the algebraic tensor product $\mathcal B(\mathcal H)\odot\mathcal B^1(\mathcal Z)$.
However, it is not clear what topology on the domain would allow for a completion of this correspondence to $\mathcal B(\mathcal B^1(\mathcal H),\mathcal B^1(\mathcal Z))$: the norm topology leads to a domain $\mathcal K(\mathcal B^1(\mathcal H),\mathcal B^1(\mathcal Z))$ which is too small, but something like the weak operator topology would likely be too weak and the domain would become too large.

Another question is concerned with the unbounded case, that is, the case of dynamical semigroups which are not norm- but only strongly continuous. While there are certain models where one can make sense of expressions like ${\rm tr}(B^*H)$ with $H$ unbounded (the relevant notion here are so-called ``Schwartz operators'', cf.~\cite{KMW16}{}), the bigger problem is that
in the strongly continuous case there is no standard form of the corresponding generator anymore, refer to \cite{OSID_Werner_17}{}. 

All questions posed above would, if solved, most likely require a vastly different set of tools as well as
deep new insights into the Choi formalism or into generators of completely positive semigroups which is why we pointed them out explicitly.

\section{Appendix: Auxiliary Lemmata}
The following result establishes how the strong operator topology interacts with the Schatten norms and is a generalization of \cite[Prop.~2.1]{Widom76} from sequences to uniformly bounded nets, although the proof stays basically the same:
\begin{lemma}\label{lemma_strong_op_schatten}
Let arbitrary complex Hilbert spaces $\mathcal H,\mathcal Z$ and $A\in\mathcal B^p(\mathcal H,\mathcal Z)$ for $p\in[1,\infty]$, as well as nets $(B_i)_{i\in I}\subset\mathcal B(\mathcal H)$, $(C_j)_{j\in J}\subset\mathcal B(\mathcal Z)$ be given. If there exist bounded operators $B,C\in\mathcal B(\mathcal H)$ such that $\|B_ix-Bx\|,\|C_iy-Cy\|\to 0$ for all $x,y\in\mathcal H$, and if there exists $\kappa>0$ such that $\|B_i\|_\infty\leq\kappa$, $\|C_j\|_\infty\leq\kappa$ for all $i\in I, j\in J$,
then the net $(B_iAC_j^*)_{i\in I,j\in J}\subset\mathcal B^p(\mathcal H,\mathcal Z)$ converges to $BAC^*$ in $p$-norm.
If $I=J$, then $(B_iAC_i^*)_{i\in I}\subset\mathcal B^p(\mathcal H,\mathcal Z)$ converges to $BAC^*$ in $p$-norm, as well.
In particular this holds for all sequences (i.e.~$I,J\subseteq\mathbb N$) which strongly converge to some bounded operator, even without the above uniform boundedness requirement.
\end{lemma}
\begin{proof}
Consider any Schmidt decomposition $\sum_{k=1}^\infty s_k(A)|e_k\rangle\langle f_k|$
and let $\varepsilon>0$.
W.l.o.g.~$\|B\|_\infty,\|C\|_\infty<\kappa$---else define $\tilde\kappa:=\max\{\kappa,\|B\|_\infty,\|C\|_\infty\}<\infty$. Now there exists $N\in\mathbb N$ such that
$$
\begin{cases}
\sum_{k=N+1}^\infty s_k(A)^p<\frac{\varepsilon^p}{(3\kappa^2)^p}&\text{ if }p\in[1,\infty)\\
s_{N+1}(A)<\frac{\varepsilon}{3\kappa^2}&\text{ if }p=\infty
\end{cases}\,.
$$
Then, using the directed set property, strong convergence of $(B_i)_{i\in I}$ implies the existence of $i_0\in I$ such that
$
\|B_ie_k-Be_k\|<\frac{\varepsilon}{6\kappa\sum_{k=1}^N s_k(A)}
$
for all $i\succeq i_0$ and all $k=1,\ldots,N$;
strong convergence of $(C_j)_{j\in J}$ on the set $\{f_1,\ldots,f_N\}$ yields a similar $j_0\in J$. Defining $A_1:=\sum_{k=1}^Ns_k(A)|e_k\rangle\langle f_k|$ and $A_2:=A-A_1$ we for all $(i,j)\succeq (i_0,j_0)$
compute
\begin{align*}
\|B_i&AC_j^*-BAC^*\|_p\\
&\leq\|B_iA_1C_j^*-B_iA_1C_j^*\|_p+\|B_iA_2C_j^*\|_p+\|BA_2C^*\|_p\\
&<\Big\|\sum_{k=1}^N s_k(A)|(B_i-B)e_k\rangle\langle C_jf_k|\Big\|_p+\Big\|\sum_{k=1}^N s_k(A)|Be_k\rangle\langle (C_j-C)f_k|\Big\|_p+\frac{2\varepsilon}{3}\\
&\leq\kappa  \sum_{k=1}^N s_k(A)\big( \|B_ie_k-Be_k\|+\|C_jf_k-Cf_k\| \big)+\frac{2\varepsilon}{3}<\varepsilon\,.
\end{align*}
The case $I=J$ (i.e.~$\|B_iAC_i^*-BAC^*\|_p\to 0$) is shown analogously. Now the additional statement 
follows at once from the
uniform boundedness principle.
\end{proof}
This has an immediate consequence for block approximations of Schatten class as well as for general bounded operators:
\begin{corollary}\label{coro_block_approx_nonsep}
Let arbitrary complex Hilbert spaces $\mathcal H,\mathcal Z$
and $A\in\mathcal L(\mathcal H,\mathcal Z)$ be given.
For any orthonormal bases $\{f_k\}_{k\in K}$, $\{g_j\}_{j\in J}$ of $\mathcal Z$, $\mathcal H$, respectively, as well as any finite subsets $J'\subseteq J$, $\mathcal K'\subseteq K$ define
$$
A_{J',K'}:=\sum_{j\in J'}\sum_{k\in K'}\langle  f_k,Ag_j\rangle |f_k\rangle\langle g_j|\in
\mathcal B(\mathcal H,\mathcal Z)\,.
$$
Then the following statements hold.
\begin{itemize}
\item[(i)] If $A\in\mathcal B^p(\mathcal H)$ for some $p\in[1,\infty]$,
then $(A_{J',K'})_{J'\subseteq J,\mathcal K'\subseteq K\text{ finite}}$ converges to $A$ in $p$-norm.
\item[(ii)] If $A$ is bounded, then $(A_{J',K'})_{J'\subseteq J,\mathcal K'\subseteq K\text{ finite}}$ converges to $A$ in the strong operator topology.
\end{itemize}
In particular, for all orthonormal bases $\{f_k\}_{k\in K}$, $\{g_j\}_{j\in J}$ of $\mathcal Z$, $\mathcal H$, respectively,
${\rm span}\{ |f_k\rangle\langle g_j|:j\in J,k\in K\}$ is dense in $(\mathcal B^p(\mathcal H,\mathcal Z),\|\cdot\|_p)$ for all $p\in[1,\infty]$.
\end{corollary}
\begin{proof}
(i): Given any finite subsets $J'\subseteq J$, $K'\subseteq K$, respectively, define the orthogonal projections $\Pi_{J'}:=\sum_{j\in J'}| g_j\rangle\langle  g_j|$, $\tilde\Pi_{K'}:=\sum_{k\in K'}| f_k\rangle\langle  f_k|$.
The corresponding nets $(\Pi_{J'})_{J'\subseteq J \text{ finite}}$, $(\tilde\Pi_{K'})_{K'\subseteq K\text{ finite}}$
are well known to
converge
to ${\bf1}_{\mathcal H}$, ${\bf1}_{\mathcal Z}$, respectively, in the strong operator topology.
With this, Lemma \ref{lemma_strong_op_schatten} shows $\|A_{J',K'}-A\|_p=\|\tilde\Pi_{K'}A\Pi_{J'}-A\|_p\to 0$ because $\|\Pi_{J'}\|_\infty=\|\tilde\Pi_{K'}\|_\infty=1$ for all $J'\subseteq J,K'\subseteq K$.
The additional statement also follows from this.
(ii): As before,
for all $x\in\mathcal H$
$
\|\tilde\Pi_{K'}A\Pi_{J'}x-Ax\|
\leq \|A\|_\infty\|(\Pi_{J'}-{\bf1}_{\mathcal H})x\|+\|(\tilde\Pi_{K'}-{\bf1}_{\mathcal Z})Ax\|\to 0\,.
$
\end{proof}

Finally, we prove a technical lemma about trace norm convergence of certain sums of operators.

\begin{lemma}\label{lemma_strongconv_kraus_app}
Let complex Hilbert spaces $\mathcal H,\mathcal Z$ and $\{V_j\}_{j\in J}\subset\mathcal B(\mathcal H,\mathcal Z)$ be given such that
$\{\sum_{j\in F}V_j^*V_j\}_{F\subseteq J\text{ finite}}$ is uniformly bounded.
Then $\{\sum_{j\in F}V_jAV_j^*\}_{F\subseteq J\text{ finite}}$ converges in trace norm for all $A\in\mathcal B^1(\mathcal H)$, $A\geq 0$.
Moreover, if $X\in\mathcal B(\mathcal H)$ denotes the limit of $\{\sum_{j\in F}V_j^*V_j\}_{F\subseteq J\text{ finite}}$, then
${\rm tr}(\sum_{j\in J}V_jAV_j^*)={\rm tr}(XA)$.
\end{lemma}
\begin{proof}
Given $A\in\mathcal B^1(\mathcal H)$ positive semi-definite (so $A=\sum_{k\in N}s_k(A)|g_k\rangle\langle g_k|$ for some orthonormal system $\{g_k\}_{k\in N}$ in $\mathcal H$)
our goal is to show that for all $\varepsilon>0$ there exists $F_\varepsilon\subseteq J$ finite such that $\|\sum_{j\in F}V_jAV_j^*\|_1<\varepsilon$ for all $F\subseteq J\setminus F_\varepsilon$ finite (this is sufficient due to Lemma~1.2.2 in \cite{Ringrose71}{}).
By Lemma~\ref{lemma_krausopconv} $\{\sum_{j\in F}V_j^*V_j\}_{F\subseteq J\text{ finite}}$ converges strongly and ultraweakly---so in particular weakly---to some $X\in\mathcal B(\mathcal H)$.
W.l.o.g.~$X\neq 0$ (else $V_j=0$ for all $j\in J$) as well as $A\neq 0$.

Now let any $\varepsilon>0$ be given. Because $(s_k(A))_{k\in N}\in\ell^1(N,\mathbb C)$,
there exists $N_\varepsilon\subseteq N$ such that $ N\setminus N_\varepsilon$ is finite and that $\sum_{k\in N\setminus N_\varepsilon}s_k(A)>0$, as well as
$
\sum_{k\in N_\varepsilon}s_k(A)<\frac{\varepsilon}{2\|X\|_\infty}
$.
Moreover, because $\{\sum_{j\in F}V_j^*V_j\}_{F\subseteq J\text{ finite}}$ converges weakly,
$\{\langle g_k,\sum_{j\in F}V_j^*V_jg_k\rangle\}_{F\subseteq J\text{ finite}}=\{\sum_{j\in F}\|V_jg_k\|^2\}_{F\subseteq J\text{ finite}}$ is summable for all $k\in N\setminus N_\varepsilon$.
This way one iteratively finds $F_\varepsilon\subseteq J$ finite such that $\sum_{j\in F}\|V_jg_k\|^2<\frac{\varepsilon}{2\sum_{k\in N\setminus N_\varepsilon}s_k(A)}$ for all $F\subseteq J\setminus F_\varepsilon$ finite and all $k\in N\setminus N_\varepsilon$.
This is all we need in order to verify the above Cauchy criterion: for all $F\subseteq J\setminus F_\varepsilon$ finite, using the triangle inequality we
can upper bound $\|\sum_{j\in F}V_jAV_j^*\|_1$ via
\begin{align*}
\sum_{j\in F}&\sum_{k\in N\setminus N_\varepsilon}s_k(A)\big\|V_j|g_k\rangle\langle g_k|V_j^*\big\|_1+ \sum_{j\in F}\sum_{k\in N_\varepsilon}s_k(A)\big\|V_j|g_k\rangle\langle g_k|V_j^*\big\|_1\\
&\leq \sum_{k\in N\setminus N_\varepsilon}s_k(A)\Big(\sum_{j\in F}\|V_jg_k\|^2\Big)+\sum_{k\in  N_\varepsilon}s_k(A)
\Big\langle g_k,\sum_{j\in F}V_j^*V_j  g_k\Big\rangle
\\
&<\frac{\varepsilon}{2}+\sum_{k\in  N_\varepsilon}s_k(A)\Big\|\sum_{j\in F}V_j^*V_j \Big\|_\infty\leq \frac{\varepsilon}{2}+\sum_{k\in  N_\varepsilon}s_k(A)\|X\|_\infty<\frac{\varepsilon}{2}+\frac{\varepsilon}{2}=\varepsilon\,.
\end{align*}
The final claim follows from the trace-norm convergence we just showed, together with ultraweak convergence of $\{\sum_{j\in F}V_j^*V_j\}_{F\subseteq J\text{ finite}}$:
\begin{align*}
{\rm tr}\Big(\sum_{j\in J}V_jAV_j^*\Big)=\lim_F{\rm tr}\Big(\sum_{j\in F}V_jAV_j^*\Big)=\lim_F{\rm tr}\Big(\sum_{j\in F}V_j^*V_jA\Big)={\rm tr}(XA)
\end{align*}
\end{proof}

\section*{Acknowledgment}
I appreciate fruitful discussions with Federico Girotti at the ``Workshop on dissipative evolutions on infinite dimensional and many-body quantum systems'' 2023 in Nottingham.
There it was pointed out to me that injectivity of $B$---which was an assumption in Prop.~\ref{prop_CPB_diff_impossible} ff.~in an earlier draft---is not necessary for the main results to hold.
Moreover, I am grateful to the anonymous referee for their constructive comments, as well as to Sumeet Khatri for pointing me to some references on the infinite-dimensional Choi formalism in quantum optics.
This work has been supported by the Einstein Foundation (Einstein Research Unit on Quantum Devices) and the MATH+ Cluster of Excellence.

%
%
%
%
%
%
%

  \bibliographystyle{ws-idaqp} 
  \bibliography{../../../../../../../control21vJan20.bib}
\end{document}